\documentclass{amsart}

\usepackage{amssymb}
\usepackage{url}
\usepackage{amscd}
\usepackage[margin=2.5cm]{geometry}
\allowdisplaybreaks

\newcommand{\intM}[1]{\int_{\Sigma}{{#1}d\mu}}
\newcommand{\llll}[1]{\left\|{#1}\right\|}
\newcommand{\oo}[1]{\left({#1}\right)}
\newcommand{\cc}[1]{\left[{#1}\right]}
\newcommand{\ben}{\begin{eqnarray*}}
\newcommand{\een}{\end{eqnarray*}}
\newcommand{\inner}[1]{\left\langle{#1}\right\rangle}
\newcommand{\cgam}{c_{\gamma}}
\newcommand{\vn}[1]{\lVert#1\rVert}
\newcommand{\ctild}{c_{\tilde{\gamma}}}

\newcommand{\cmss}{c_{\textrm{M\hskip-0.1mm S\hskip-0.1mm S}}}
\newcommand{\R}{\mathbb{R}}
\newcommand{\norm}[1]{\left|{#1}\right|}
\newcommand{\nablanorm}[3]{\left|\nabla_{\oo{#1}}{#2}\right|^{#3}}
\newcommand{\Ao}[1]{\left\|A^{o}\right\|_{2,\cc{\gamma>0}}^{#1}}
\newcommand{\A}[1]{\left\|A\right\|_{2,\cc{\gamma>0}}^{#1}}

\newcommand{\Kappa}{\alpha}
\newcommand{\snabla}{\overline{\nabla}}
\newcommand{\sDelta}{\overline{\Delta}}

\theoremstyle{plain}
\newtheorem{theorem}{Theorem}

\newtheorem{corollary}[theorem]{Corollary}
\newtheorem{lemma}[theorem]{Lemma}
\newtheorem{proposition}[theorem]{Proposition}

\theoremstyle{remark}

\newtheorem{rmk}{Remark}

\begin{document}

\date{****Received 6 December 2013; received in final form *******} 

\title{The geometric triharmonic heat flow of immersed surfaces near spheres}

\author[J. McCoy]{James McCoy}
\address{Institute for Mathematics and its Applications, School of Mathematics and Applied Statistics\\
University of Wollongong\\
Northfields Ave, Wollongong, NSW $2500$,\\
Australia}
\email{jamesm@uow.edu.au}

\author[S. Parkins]{Scott Parkins}
\address{Institute for Mathematics and its Applications, School of Mathematics and Applied Statistics\\
University of Wollongong\\
Northfields Ave, Wollongong, NSW $2500$,\\
Australia}
\email{srp854@uow.edu.au}

\author[G. Wheeler]{Glen Wheeler}
\address{Institute for Mathematics and its Applications, School of Mathematics and Applied Statistics\\
University of Wollongong\\
Northfields Ave, Wollongong, NSW $2500$,\\
Australia}
\email{glenw@uow.edu.au}

\thanks{
The research of the second author was supported by an Australian
Postgraduate Award.  The research of the first and third authors was supported
by Discovery Project DP120100097 of the Australian Research Council.
Results from this paper will appear in the PhD thesis of the second author.
}


\begin{abstract}
We consider closed immersed surfaces in $\mathbb{R}^{3}$ evolving by the
geometric triharmonic heat flow.
Using local energy estimates, we prove interior estimates and a positive
absolute lower bound on the lifespan of solutions depending solely on the local
concentration of curvature of the initial immersion in $L^2$.
We further use an $\varepsilon$-regularity type result to prove a gap lemma for
stationary solutions.
Using a monotonicity argument, we then prove that a blowup of the flow
approaching a singular time is asymptotic to a non-umbilic embedded stationary
surface.
This allows us to conclude that any solution with initial $L^2$-norm of the
tracefree curvature tensor smaller than an absolute positive constant converges
exponentially fast to a round sphere with radius equal to
$\sqrt[3]{3V_0/4\pi}$, where $V_0$ denotes the signed enclosed volume of the
initial data.
\end{abstract}

\maketitle

\begin{section}{Introduction}

Let $f:\Sigma\times\left[0,T\right)\rightarrow\mathbb{R}^{3}$, $T>0$, be a
one-parameter family of compact immersed surfaces
$f\oo{\cdot,t}=f_{t}:\Sigma\rightarrow f_{t}\oo{\Sigma}=\Sigma_{t}$.
We say $f$ is evolving under the geometric triharmonic heat flow if it satisfies the
following equation
\begin{equation}
\frac{\partial}{\partial t}f=-(\Delta^{2}H)\nu\,.\label{trilaplacianIntro1}\tag{GTHF}
\end{equation}
with given smooth initial surface $f\left( \cdot, 0 \right) = f_0$.
Above $\Delta$ and $H$ are respectively the Laplace-Beltrami operator and mean
curvature of the surface $\Sigma_{t}$, and $\nu$ is the outer unit normal.
Further notation and conventions are set out in Section 2.

The mean curvature vector $\vec{H}$ satisfies $\Delta f = \vec{H} = -H\nu$.
As $\vec{H}$ is a section of the normal bundle, it is natural to apply to this
the induced Laplacian in the normal bundle $\Delta^\perp$.
In this case one recovers $(\Delta^2 H)\nu = \Delta^\perp\Delta^\perp\Delta f$.
This is natural from a geometric perspective as well as from the perspective of
curvature flow: only normal terms in the speed of the flow affect geometric
quantities which are invariant under the diffeomorphism group, and so any
natural geometric operator should have image in the normal bundle.
One may however wish to interpret $\vec{H}$ as a vector in $\R^3$ and instead
apply the rough Laplacian $\Delta_R$ to it.
In this case one recovers a lengthy expression for $\Delta_R\Delta_R\Delta f$
which has leading order term $(\Delta^2 H)\nu$ with many other terms
contributing in tangential and normal directions.
We have termed \eqref{trilaplacianIntro1} the geometric triharmonic heat flow
in order to distinguish it from the possible second interpretation.

As \eqref{trilaplacianIntro1} is a sixth-order quasilinear degenerate parabolic
system of partial differential equations, the local existence of a solution is
standard (see, eg. \cite{Mantegazza1,Mantegazza2}) for regular enough initial
data.

\begin{theorem}
\label{TMste}
Suppose $\Sigma$ is a complete, compact 2-manifold without boundary.
Let $f_0:\Sigma\rightarrow\R^3$ be a $C^{6,\alpha}$ isometric immersion.
There exists a maximal $T>0$ and a corresponding unique one-parameter family of smooth isometric immersions $f:\Sigma\times(0,T)\rightarrow\R^3$ satisfying \eqref{trilaplacianIntro1} and
\[
f(\cdot,t) \rightarrow f_0(\cdot)\quad\text{  as  }\quad t\searrow0\,,
\]
locally smoothly in the $C^{6,\alpha}$ topology.
\end{theorem}

In this paper we will always assume that the initial data $f_0$ is smooth, and
so then the flow generated by $f_0$ is $f:\Sigma\times[0,T)\rightarrow\R^3$,
which is a smooth family on the half-open interval $[0,T)$.

Applications of equations involving a thrice-iterated Laplacian are growing,
and include for example the modelling of ulcers \cite{UW05}, computer graphics
\cite{BWH95,U04}, and interactive design \cite{KUW04}.  In each of these
applications it is important that the Laplacian considered possesses some
inherent geometric invariances, with the natural candidate being the
Laplace-Beltrami operator.
We hope in future work to generalise equation \eqref{trilaplacianIntro1} to
include lower order terms; such sixth order equations appear as the phase field
crystal equation proposed as a model for microstructure evolution of two-phase
systems on atomic length and diffusive time scales (see \cite{GomezNogueira}
and the references contained therein) and in surface modelling \cite{LiuXu,
Tosun}.  

In Section 6 we study properly immersed surfaces which are not necessarily closed.
Our main result is the following tracefree curvature estimate, which holds quite generally.
In particular, the immersion does not need to satisfy $\Delta^2H = 0$.

\begin{theorem}
\label{Tracefreetheorem1}
Let $x\in\R^3$, $\rho>0$, and $f:\Sigma\rightarrow\mathbb{R}^{3}$ be a locally $W^{5,2}$ immersion satisfying
\begin{equation}
\int_{f^{-1}(B_{2\rho}(x))}{\norm{A^{o}}^{2}\,d\mu}\leq\varepsilon_{0}\label{TracefreesmallnessBALL}
\end{equation}
for $\varepsilon_{0}>0$ sufficiently small.
Then there is a universal constant $c>0$ such that
\begin{equation*}
\llll{A^{o}}_{\infty,f^{-1}(B_\rho(x))}^{6}
\leq c\vn{A^o}^4_{2,f^{-1}(B_{2\rho}(x))}
  \Big(\vn{\nabla\Delta H}^2_{2,f^{-1}(B_{2\rho}(x))} + \rho^{-6}\Big)\,.
\end{equation*}
\end{theorem}

This has as an immediate corollary the following higher-order analogue of
$\varepsilon$-regularity (see \cite{Kuwert1,Schoen,Wheeler2} for the
corresponding result for the Willmore operator, minimal surfaces, and the
surface diffusion operator respectively).

\begin{corollary}
\label{CYepsreg}
Let $f:\Sigma\rightarrow\mathbb{R}^{3}$ be a locally $W^{5,2}$ closed immersion
satisfying $\Delta^2H = 0$ weakly.
Suppose that for $x\in\R^3$, $\rho>0$, the local smallness condition
\eqref{TracefreesmallnessBALL} is satisfied for $\varepsilon_0>0$ sufficiently
small.
Then there is an absolute constant $c>0$ such that
\begin{equation*}
\llll{A^{o}}_{\infty,f^{-1}(B_\rho(x))}^{6}
\leq c\,\varepsilon^2\rho^{-6}\,.
\end{equation*}
\end{corollary}

This estimate is a key first step in a larger regularity program for
weak solutions, improving the a-priori regularity assumption dramatically under
the natural assumption that the $L^2$-norm of the tracefree curvature tensor
$A^o$ is globally bounded.

If the $L^2$-norm of $A^o$ is in fact globally small, then we may apply again
Theorem \ref{Tracefreetheorem1} to obtain the following much stronger result
than Corollary \ref{CYepsreg} above.
It is proven in Section 8.

\begin{theorem}[Gap Lemma]
\label{GapTheorem1}
Suppose $f:\Sigma\rightarrow\mathbb{R}^{3}$ is a proper locally $C^6$ immersion with $\Delta^{2}H\equiv0$.
Then, if $f$ is closed and satisfies
\begin{equation}
\label{GLA1}
\intM{\norm{A^{o}}^{2}}<\varepsilon_{0}<8\pi
\end{equation}
for $\varepsilon_{0}>0$ sufficiently small, there exist $x\in\R^3$ and $\rho>0$ such that
\begin{equation*}
f\oo{\Sigma}=\mathbb{S}^{2}_{\rho}(x),
\end{equation*}
where $\mathbb{S}^{2}_\rho(x)$ denotes a standard round sphere in $\mathbb{R}^{3}$ with radius $\rho$ and centre $x\in\R^3$.
If $f$ is not closed, then under assumption \eqref{GLA1} we instead conclude that
\begin{equation*}
f\oo{\Sigma}=\mathbb{P}^{2},
\end{equation*}
where $\mathbb{P}^{2}$ denotes a standard flat plane in $\mathbb{R}^{3}$.
\end{theorem}

Gap Lemmata are known in a variety of contexts: for the Willmore operator
\cite{Kuwert1}, the surface diffusion operator \cite{Wheeler2}, and a family of
fourth-order geometric operators \cite{Wheeler4}.
We expect Theorem \ref{Tracefreetheorem1} to enjoy further applications.

Using $\text{Vol}\oo{\Sigma_{t}}$ and $\norm{\Sigma_{t}}$ to denote the signed
enclosed volume and surface area respectively of $\Sigma_{t}$ in
$\mathbb{R}^{3}$, we compute that under the flow \eqref{trilaplacianIntro1}
\begin{equation}
\label{EQvolevo}
\frac{d}{dt}\text{Vol}\oo{\Sigma_{t}}=-\intM{\Delta^{2}H}=0
\end{equation}
and
\begin{equation}
\label{EQareaevo}
\frac{d}{dt}\norm{\Sigma_{t}}=-\intM{H\Delta^{2}H}=-\intM{\norm{\Delta H}^{2}}\leq0 \mbox{.}
\end{equation}
From \eqref{EQvolevo}, \eqref{EQareaevo}, the flow \eqref{trilaplacianIntro1}
is isoperimetrically natural.
The flow may develop singularities while improving the isoperimetric ratio,
manifested as curvature singularities.
Our first main parabolic result is the following characterisation of the
singular time.

\begin{theorem}[Lifespan Theorem]\label{LifespanTheorem}
Suppose $f:\Sigma\times\left[0, T\right)\longrightarrow\mathbb{R}^{n}$
satisfies \eqref{trilaplacianIntro1} with smooth initial data $f_0$.
Then there are constants $\rho>0,\varepsilon_{0}>0$ and $C<\infty$ such that if
$\rho$ is chosen with 
\begin{equation}
\int_{f^{-1}\oo{B_{\rho}\oo{x} }}{\norm{A}^{2}d\mu}\Bigg|_{t=0}=\varepsilon\oo{x}\leq\varepsilon_{0}\text{  for any }x\in\mathbb{R}^{3},\label{Lifespansmallness}
\end{equation}
then the maximal time $T$ of smooth existence for the flow satisfies
\begin{equation}
T\geq\frac{1}{C}\rho^{6},\label{lifespan1}
\end{equation}
and we have the estimate
\begin{equation}
\int_{f^{-1}\oo{B_{\rho}\oo{x} }}{\norm{A}^{2}d\mu}\leq C\varepsilon_{0}\text{  for }0\leq t\leq\frac{1}{C}\rho^{6}.\label{lifespan2}
\end{equation}
\end{theorem}

By linearising \eqref{trilaplacianIntro1} around a round sphere, one finds that
the spectrum is entirely real and non-positive.
Using standard centre-manifold methods, one can conclude that any solution with
initial data close to a sphere $S_1$ in $C^{6,\alpha}$ exists for all time and
converges exponentially fast to a sphere $S_2$ in the $C^\infty$ topology.
Note that $S_1$ need not equal $S_2$ in any sense: the perturbation could be a
small translation, in which case the centre changes, or it may change the
enclosed volume of the initial data from that of $S_1$, in which case the
radius changes.
Using completely different methods, we are able to improve this statement,
weakening the initial condition to geometric averaged closeness in $L^2$.

Let us briefly detail the argument.
Note that if the initial immersion is locally smooth with finite total
curvature then for any $\varepsilon_{0}>0$ it is always possible to find a
positive $\rho=\rho\oo{\varepsilon_{0},\Sigma_{0}}$ such that assumption
$\oo{\ref{Lifespansmallness}}$ is satisfied.
Theorem \ref{LifespanTheorem} is a characterisation of singularities in the
following sense.
It tells us that the only way the flow can cease to exist and lose regularity
in finite time $T<\infty$ is if we encounter a specific type of curvature
singularity: if $\rho\left( t\right)$ denotes the largest radius such that
\eqref{Lifespansmallness} holds at time $t$, then $\rho\left( t\right) \leq
\sqrt[6]{C\left( T-t\right)}$ so at least $\varepsilon_0$ of the curvature
concentrates in a ball $f^{-1}\left( B_{\rho\left( T\right)}\left( x\right)
\right)$.
That is,
$$\lim_{t\rightarrow T} \int_{f^{-1}\left( B_{\rho\left( t\right)}\left( x\right) \right)} \left| A \right|^2 d\mu \geq \varepsilon_0$$
where $x=x\left( t\right)$ is the centre of a ball where the integral is maximised at time $t$.

By considering a sequence of rescalings, we are able to extract a convergent
subsequence asymptotic to a stationary solution of \eqref{trilaplacianIntro1}.
This relies on the monotonicity of the $L^2$-norm of the tracefree curvature
tensor.
The stationary solution extracted is entire, with small total tracefree
curvature in $L^2$, and so by applying the Gap Lemma we are able to show
that this is a contradiction.
Subconvergence of the flow to a specific sphere results, which we then
strengthen by a standard linearised stability argument.
Our final result is the following.

\begin{theorem}[Long time existence and exponential convergence to round spheres]
\label{Theorem1}
There exists an absolute constant $\varepsilon_{0}>0$ such that a geometric
triharmonic heat flow $f:\Sigma\times\left[0,T\right)\rightarrow\mathbb{R}^{3}$
with smooth initial data satisfying
\begin{equation}
\label{Theorem1,1}
\intM{\norm{A^{o}}^{2}}\Big|_{t=0}\leq\varepsilon_{0}<8\pi
\end{equation}
exists for all time, is a family of embeddings, and for some $x\in\R^3$ the
family $\Sigma_{t}$ converges to $\mathbb{S}_{\sqrt[3]{3V_0/4\pi}}^{2}(x)$
exponentially fast in the $C^\infty$ topology, where $V_0>0$ is the signed
enclosed volume of the initial immersion.
\end{theorem}

The paper is organised as follows.
After fixing notation in Section 2 and stating evolution equations and
interpolation inequalities in Section 3, we establish local control on the
$L^2$-norm of iterated covariant derivatives of the curvature in Section 4.
These provide us with all the requisite tools to prove the Theorem
\ref{LifespanTheorem}, which we do in Section 5.
Section 6 is devoted to elliptic analysis, and culminates in the proof of
Theorem \ref{Tracefreetheorem1}.
Section 7 establishes that $\vn{A^o}_2^2$ is a Lyapunov functional for the flow
if initially small enough.
Section 8 contains the proof of the Gap Lemma.
In Section 9 we use a compactness theorem and most of our earlier work to
establish the existence and several key properties of the blowup of a singular
time for the flow.
This is applied in Section 10 to conclude global existence and subconvergence
to a sphere, which is then strengthened by a standard linearised stability
argument.

\end{section}

\begin{section}{Preliminaries}

We consider a surface $\Sigma$ immersed into $\R^3$ via a smooth immersion
$f:\Sigma\rightarrow\mathbb{R}^{3}$.
The induced metric $g$ on $\Sigma$ is given by pulling back the standard
Euclidean inner product on $\mathbb{R}^{3}$ along $f$. It is defined pointwise
by
\begin{equation}
g_{ij}=\inner{\partial_{i}f,\partial_{j}f}.\label{Prelimiaries1}
\end{equation}
Here $\partial$ denotes the coordinate derivatives on $\Sigma$ and
$\inner{\cdot,\cdot}$ the standard Euclidean inner product. This induces an
inner product on all tensors along $f$ of similar type, which
are defined via traces over pairs of indices. For example, the inner product on
the $\oo{1,2}$-tensors $S,T$ is defined by
\[
\inner{S,T}=g^{jm}g^{kn}g_{il}S_{jk}^{i}T_{mn}^{l}.
\]
Here we adopt the Einstein convention: repeated indices are summed from $1$
to $2$.
The norm squared of a tensor $T$ is the inner product of $T$ with itself and is
denoted $\norm{T}^{2}$.
For tensors $S$ and $T$ we will also frequently utilise the notation of
Hamilton \cite{Hamilton1}, using $S\star T$ to denote an arbitrary linear
combination of new tensors that is formed by contracting pairs of indices of
$S$ and $T$ by $g$. For example, $\inner{S,T}=S\star T$. A useful property of
$\star-$notation is that
\[
S\star T\leq c\norm{S}\norm{T}
\]
for some constant $c$ that depends only on the number of indices in $S$ and $T$ combined (recall the dimension and codimension are fixed).
For a tensor $T$, we define
\[
P_{j}^{i}\oo{T}=\sum_{r_{1}+\dots+r_{j}=i}\nabla_{\oo{r_{1} }}T\star\cdots\star\nabla_{\oo{r_{j} }}T,
\]
a sum of terms, each of which contains $j$ factors of $T$ with total order of
covariant derivatives equal to $i$.
This notation is particularly useful when only the number of factors of $T$ and
total number of derivatives is important.  As an example
\[
\norm{\nabla_{\oo{2}}T}^{2}=\nabla_{\oo{2}}T\star\nabla_{\oo{2}} T=P_{2}^{4}\oo{T}.
\]
The mean curvature is defined to be
\[
H=g^{ij}A_{ij}=A_{i}^{i}
\]
where $A_{ij}$ are the components of the second fundamental form $A$:
\[
A_{ij}=-\inner{\partial_{ij}f,\nu}.
\]
The Gauss curvature $K$ is the determinant of the Weingarten map $A_i^j := g^{ik}A_{jk}$,
\[
K = \text{det }A_i^j\,.
\]
We also define the tracefree second fundamental form to be the symmetric
tracefree part of $A$, denoted $A^{o}$, with components 
\[
A_{ij}^{o}=A_{ij}-\frac{1}{2}g_{ij}H.
\]
A short calculation shows that
\begin{equation}
\oo{\nabla^{*}A^{o}}_{j}:=\nabla^{i}A_{ij}^{o}=g^{ip}\nabla_{p}\oo{A_{ij}-\frac{1}{2}g_{ij}H}=\frac{1}{2}\nabla_{j}H;
\label{EQabove10}
\end{equation}
note that $-\nabla^{*}$ is the geometric divergence operator or formal adjoint of $\nabla$ in $L^2$.
Here we have used the total symmetry of the $\oo{0,3}$-tensor $\nabla A$, known as the Codazzi equations:
\[
\nabla_{i}A_{jk}=\nabla_{j}A_{ki}=\nabla_{k}A_{ij}.
\]
It follows that
\begin{equation*}
\norm{\nabla H}^{2}= 4\norm{\nabla^{*}A^{o}}^{2}\leq4\norm{\nabla A^{o}}^{2}.
\end{equation*}
Moreover, for any $k\in\mathbb{N}$ we have
\begin{equation}
\nablanorm{k}{H}{2}\leq4\nablanorm{k}{A^{o}}{2}\label{Prelimiaries2}
\end{equation}
and hence
\begin{equation}
\nablanorm{k}{A}{2}=\nablanorm{k}{A^{o}}{2}+\frac{1}{2}\nablanorm{k}{H}{2}\leq3\nablanorm{k}{A^{o}}{2}.\nonumber
\end{equation}
The Laplace-Beltrami operator acts on the components of an $\oo{m,n}$-tensor $S$ via
\[
\Delta S_{j_{1}\dots j_{n}}^{i_{1}\dots i_{m}}=g^{pq}\nabla_{pq}S_{j_{1}\dots j_{n}}^{i_{1}\dots i_{m}}.
\]
We define the integral of a compactly supported function
$h:\Sigma\rightarrow\mathbb{R}$ as 
\[
\intM{h\,},
\]
where $d\mu$ is the induced measure on $\Sigma$:
\[
d\mu:=\sqrt{\det\left( g_{ij} \right)}\, d\mathcal{L}^{2}.
\]
We denote the area by
\[
\mu(\Sigma) = |\Sigma| = \int_\Sigma\,d\mu\,.
\]
We will frequently employ the Divergence Theorem to `integrate by parts' over a
non-compact immersion.
To do so we include a cut-off function $\gamma:\Sigma\rightarrow\mathbb{R}^{3}$ in the integrand.
We keep this function arbitrary but sufficiently smooth with compact support so
that we may eventually conclude global results from results that hold on the
support of $\gamma$.
In particular, we take $\gamma=\tilde{\gamma}\circ f:\Sigma\rightarrow\cc{0,1}$
for some $C^{3}$ function $\tilde{\gamma}$ satisfying
\[
0\leq\tilde{\gamma}\leq1\text{  and  }\llll{\tilde{\gamma}}_{C^{3}\oo{\mathbb{R}^{3} }}\leq\ctild<\infty.
\]
Using the chain rule and the fact that $\left\{\partial_{i}f\right\}$ provides
a basis for the tangent bundle $T\Sigma$, there is a universal, bounded
constant $\cgam>0$ such that $\gamma$ satisfies
\begin{equation}
\llll{\nabla\gamma}_{\infty}\leq\cgam,\llll{\nabla_{\oo{2}}\gamma}_{\infty}\leq\cgam\oo{\cgam+\norm{A}},\text{  and  }
\llll{\nabla_{\oo{3}}\gamma}_{\infty}\leq\cgam\oo{\cgam^{2}+\cgam\norm{A}+\norm{\nabla A}} \mbox{.}\label{E:gammaprop}
\end{equation}

\end{section}

\begin{section}{Evolution equations for elementary geometric quantities}
Under \eqref{trilaplacianIntro1}, we have the following evolution equations for geometric quantities associated to $\Sigma_t$.
The proof of Lemma \ref{EvolutionLemma1} is standard and straightforward.
\begin{lemma}\label{EvolutionLemma1}
Let $f:\Sigma\times\left[0,T\right)\rightarrow\mathbb{R}^{3}$ satisfy ${\eqref{trilaplacianIntro1}}$. Then
\begin{align*}
\frac{\partial}{\partial t}g=-2\oo{\Delta^2 H}A,\hspace*{.5cm}&\frac{\partial}{\partial t}d\mu=-H\oo{\Delta^{2}H}d\mu, \hspace*{.5cm}\frac{\partial}{\partial t}\nu=\nabla\Delta^{2}H,\\
\frac{\partial}{\partial t}A=\Delta^{3}A+P_{3}^{4}\oo{A},&\text{  and  } \frac{\partial}{\partial t}H=\Delta^{3}H+(\Delta^{2}H)\norm{A}^{2}.
\end{align*}
\end{lemma}
By combining the evolution equations in Lemma \ref{EvolutionLemma1} with the
interchange of covariant derivatives formula, we obtain the following evolution
equations.
\begin{lemma}\label{EvolutionLemma2}
Let $f:\Sigma\times\left[0,T\right)\rightarrow\mathbb{R}^{3}$ satisfy ${\eqref{trilaplacianIntro1}}$. Then for any $k\in\mathbb{N}_{0}$:
\[
\frac{\partial}{\partial t}\nabla_{\oo{k}}A=\Delta^{3}\nabla_{\oo{k}}A+P_{3}^{k+4}\oo{A}.
\]
\end{lemma}
\begin{corollary}\label{EvolutionCorollary1}
Let $f:\Sigma\times\left[0,T\right)\rightarrow\mathbb{R}^{3}$ satisfy ${\eqref{trilaplacianIntro1}}$. Then for any $k\in\mathbb{N}_{0}$:
\[
\frac{\partial}{\partial
t}\nablanorm{k}{A}{2}=2\inner{\nabla_{\oo{k}}A,\nabla^{i_{3}i_{2}i_{1}}\nabla_{i_{1}i_{2}i_{3}}\nabla_{\oo{k}}A}+P_{3}^{k+4}\oo{A}\star\nabla_{\oo{k}}A.
\]
\end{corollary}
Using Corollary $\ref{EvolutionCorollary1}$, Lemma $\ref{EvolutionLemma1}$, the
product rule, and integrating by parts thrice we obtain the following.
\begin{corollary}\label{EvolutionCorollary2}
Let $f:\Sigma\times\left[0,T\right)\rightarrow\mathbb{R}^{3}$ satisfy ${\eqref{trilaplacianIntro1}}$. Then for any $k,s\in\mathbb{N}_{0}$:
\begin{multline*}
\frac{d}{dt}\intM{\nablanorm{k}{A}{2}\gamma^{s}}+2\intM{\nablanorm{k+3}{A}{2}}=\intM{\nablanorm{k}{A}{2}\partial_{t}\gamma^{s}}\\
-2\sum_{j=1}^{3}\dbinom{3}{j}\intM{\inner{\oo{\nabla_{\oo{j}}\gamma^{s}}\oo{\nabla_{\oo{k+3-j}}A},\nabla_{\oo{k+3}}A}} +\intM{ \left( P_{3}^{k+4}\oo{A}\star\nabla_{\oo{k}}A\right)\gamma^{s}}.
\end{multline*}
\end{corollary}
To deal with the extraneous intermediate terms above we will make use of \eqref{E:gammaprop} and an interpolation inequality of Kuwert and Sch\"{a}tzle \cite[Corollary 5.3]{Kuwert1}. This gives us:
\begin{lemma}\label{EvolutionLemma3}
Let $2\leq p<\infty,k,m\in\mathbb{N}$ and $s\geq kp$.
Then for $\delta>0$ there exists a constant $c>0$ depending only on $\delta$, $s$ and $p$ such that
\begin{align*}
  \oo{\intM{\nablanorm{k}{A}{p}\gamma^{s} }}^{\frac{1}{p}}
\leq\delta\, \cgam^{-1}\oo{\intM{\nablanorm{k+1}{A}{p}\gamma^{s+p} }}^{\frac{1}{p}}+c_{\delta}\, \cgam^{k}\oo{\int_{\cc{\gamma>0}}{\norm{A}^{2}}\gamma^{s-kp}\,d\mu}^{\frac{1}{p}} \mbox{.}
\end{align*}
\end{lemma}
The proof of Lemma \ref{EvolutionLemma3} is essentially the same as in
\cite{Kuwert1}. However, we have retained the derivative bounds $\cgam$ to make
more explicit the scale-invariance of each quantity.  Using Lemma
\ref{EvolutionLemma3}, the properties \eqref{E:gammaprop} and Corollary
\ref{EvolutionCorollary2} we obtain the following estimate.
\begin{proposition}\label{EvolutionProposition1}
Let $f:\Sigma\times\left[0,T\right)\rightarrow\mathbb{R}^{3}$ satisfy
${\eqref{trilaplacianIntro1}}$.
For any $\delta>0$, $k\in\mathbb{N}_{0}$, and $s\geq2\oo{k+3}$ there exists a constant $c>0$ depending only on $\delta$, $s$ and $k$ such that
\begin{equation*}
  \frac{d}{dt}\intM{\nablanorm{k}{A}{2}\gamma^{s}}+\oo{2-\delta}\intM{\nablanorm{k+3}{A}{2}\gamma^{s}}
\leq\intM{\left( P_{3}^{k+4}\oo{A}\star\nabla_{\oo{k}}A \right) \gamma^{s}}+c \, \cgam^{2\oo{k+3}}\llll{A}_{2,\cc{\gamma>0}}^{2} \mbox{.}
\end{equation*}
\end{proposition}

In order to control the localised norms $\vn{\nabla_{(k)}A}^2_{2,\gamma^{s}} :=
\int_\Sigma |\nabla_{(k)}A|^2\gamma^sd\mu$ we will employ standard
interpolation and Sobolev inequalities, which we now state.

\begin{theorem}[Michael-Simon Sobolev Inequality \cite{MichaelSimon1}]\label{MichaelSimon}
Let $\Sigma$ be an immersed surface and $u\in C_{c}^{\infty}\oo{\Sigma}$. Then, for a universal, bounded constant $\cmss>0$,
\[
\intM{u^{2}}\leq\cmss\oo{\intM{\norm{\nabla u}}+\intM{\norm{u}\norm{H} }}^{2}
\] 
\end{theorem}

\begin{theorem}[$\text{\cite[Theorem 5.6]{Kuwert2}}$]\label{EvolutionTheorem1}
Let $f:\Sigma\rightarrow\mathbb{R}^{3}$ be a smooth immersion. For
$n<p\leq\infty$, $0\leq\beta\leq\infty$ and $0<\alpha\leq1$, where
$\frac{1}{\alpha}=\oo{\frac{1}{2}-\frac{1}{p}}\beta+1$, there is a constant $c$
depending on $p$ and $\beta$ such that for all $u\in C_{c}^{1}\oo{\Sigma}$,
\begin{equation}
\llll{u}_{\infty}\leq c\llll{u}_{p}^{1-\alpha}\oo{\llll{\nabla u}_{p}+\llll{Hu}_{p}}^{\alpha} \mbox{.}\label{EvolutionTheorem1,1}
\end{equation}
\end{theorem}

\begin{proposition}[$\text{\cite[Corollary 5.5]{Kuwert2}}$]\label{EvolutionProposition4}
Let $0\leq i_{1},\dots,i_{r}\leq k,\sum_{j=1}^{r}i_{j}=2k$ and $s\geq2k$. Then for any tensor $T$ defined over an immersed surface $f:\Sigma\rightarrow\mathbb{R}^{3}$ we have
\begin{equation*}
\norm{\intM{\nabla_{i_{1}}T\star\cdots\star\nabla_{i_{r}}T\gamma^{s} }}\leq c\llll{T}_{\infty,\cc{\gamma>0}}^{r-2}\oo{\intM{\nablanorm{k}{T}{2}\gamma^{s}}+\cgam^{2k}\llll{T}_{2,\cc{\gamma>0}}^{2}}.
\end{equation*}
\end{proposition}

Note that we have included $\cgam$ explicitly on the right hand side above.
This will be important later when we select a particular cutoff function.
The next estimate below is an adaptation of \cite[Lemma 4.3]{Kuwert2} to our situation.
The proof is similar.

\begin{proposition}\label{EvolutionProposition2}
Suppose $T$ is a tensor field, $s\ge6$ and $\gamma$ is as in
$\oo{\ref{E:gammaprop}}$.
There is an $\varepsilon>0$ such that
$\llll{A}_{2,\cc{\gamma>0}}^{2}\leq\varepsilon_{0}$ implies that there exists a
universal, bounded constant $c>0$ depending only on the order of $T$ such that
\begin{multline*}
\llll{T}_{\infty,\cc{\gamma=1}}^{6}\leq c\llll{T}_{2,\cc{\gamma>0}}^{4}\Biggl(\llll{\nabla_{\oo{3}}T\cdot\gamma^{\frac{s}{2} }}_{2}^{2}+\llll{\norm{\nabla T}^{2}\norm{A}^{4}\cdot\gamma^{s}}_{1}\\
+ \llll{\norm{\nabla_{\oo{2}}T}\norm{A}\cdot\gamma^{\frac{s}{2} }}_{2}^{2}+\llll{\norm{T}\norm{\nabla A}\norm{A}\cdot\gamma^{\frac{s}{2} }}_{2,}^{2}+\cgam^{6}\llll{T\cdot\gamma^{\frac{s-6}{2} }}_{2,\cc{\gamma>0}}^{2}\Biggr) \mbox{.}
\end{multline*}
\end{proposition}

The lemma below is similar to \cite[Lemma 4.2]{Kuwert2} and \cite[Proposition 2.6]{Kuwert1}.
We have incorporated the smallness assumption and since the proof involves some
new estimates, we present it.

\begin{lemma}\label{EvolutionLemma4}
Let $\gamma$ satisfy $\oo{\ref{E:gammaprop}}$. Then for an immersion $f:\Sigma\rightarrow\mathbb{R}^{3}$ satisfying 
\begin{equation}
\int_{\cc{\gamma>0}}{\norm{A}^{2}\,d\mu}\leq\varepsilon_{0}\label{Curvaturesmallness}
\end{equation}
for $\varepsilon_{0}>0$ sufficiently small, we have
\begin{multline*}
  \intM{\nablanorm{2}{A}{2}\norm{A}^{2}\gamma^{s}}+\intM{\norm{\nabla A}^{2}\norm{A}^{4}\gamma^{s}}
\leq c\int_{\cc{\gamma>0}}{\norm{A}^{2}\,d\mu}\Bigl(\intM{\nablanorm{3}{A}{2}\gamma^{s}}+\cgam^{6}\int_{\cc{\gamma>0}}{\norm{A}^{2}\,d\mu}\Bigr).
\end{multline*}
\end{lemma}
\begin{proof}
The second term can be estimated via Theorem \ref{MichaelSimon} with $u=\norm{\nabla A}\norm{A}^{2}\gamma^{\frac{s}{2}}$:
\begin{align}
&\intM{\norm{\nabla A}^{2}\norm{A}^{4}\gamma^{s}}\nonumber\\
&\leq c\Bigl(\intM{\norm{\nabla_{\oo{2}}A}\norm{A}^{2}\gamma^{\frac{s}{2} }}+\intM{\norm{\nabla A}^{2}\norm{A}\gamma^{\frac{s}{2} }}+\cgam\intM{\norm{\nabla A}\norm{A}^{2}\gamma^{\frac{s}{2}-1}}
 +\intM{\norm{\nabla A}\norm{A}^{3}\gamma^{\frac{s}{2} }}\Bigr)^{2}\nonumber\\
&\leq c\llll{A}_{2,\cc{\gamma>0}}^{2}\Bigl(\intM{\nablanorm{3}{A}{2}\gamma^{s}}+\intM{\nablanorm{2}{A}{2}\norm{A}^{2}\gamma^{s}}
+\intM{\norm{\nabla A}^{2}\norm{A}^{4}\gamma^{s}}\nonumber\\
&\quad +\intM{\norm{\nabla A}^{4}\gamma^{6}}+\cgam^{6}\A{2}\Bigr).\label{EvolutionLemma4,1}
\end{align}
Here we have used integration by parts, the Cauchy-Schwarz inequality and Lemma $\ref{EvolutionLemma3}$. For the second last term in $\oo{\ref{EvolutionLemma4,1}}$ we use integration by parts and Lemma $\ref{EvolutionLemma3}$:
\begin{align*}
  &\intM{\norm{\nabla A}^{4}\gamma^{s}}\leq3\intM{\norm{\nabla_{\oo{2}}A}\norm{\nabla A}^{2}\norm{A}\gamma^{s}}+s\, \cgam\intM{\norm{\nabla A}^{3}\norm{A}\gamma^{s-1}}\\
&\leq\frac{1}{2}\intM{\norm{\nabla A}^{4}\gamma^{s}}+c\Bigl(\intM{\nablanorm{2}{A}{2}\norm{A}^{2}\gamma^{s}}+\intM{\norm{\nabla A}^{2}\norm{A}^{4}\gamma^{s}}\Bigr)
+c \, \cgam^{4}\intM{\norm{\nabla A}^{2}\gamma^{s-4}},
\end{align*}
which implies that
\begin{multline*}
\intM{\norm{\nabla A}^{4}\gamma^{s}}\leq c\Bigl(\intM{\nablanorm{3}{A}{2}\gamma^{s}}+\intM{\nablanorm{2}{A}{2}\norm{A}^{2}\gamma^{s}}
+\intM{\norm{\nabla A}^{2}\norm{A}^{4}\gamma^{s}}+\cgam^{6}\A{2}\Bigr)\,.
\end{multline*}
For the second term above, we use Theorem \ref{MichaelSimon} with $u=\norm{\nabla_{\oo{2}}A}\norm{A}\gamma^{\frac{s}{2}}$:
\begin{multline}
  \intM{\nablanorm{2}{A}{2}\norm{A}^{2}\gamma^{s}}
   \leq c\int_{\cc{\gamma>0}}{\norm{A}^{2}\,d\mu}\Bigl(\intM{\nablanorm{3}{A}{2}\gamma^{s}} 
  +\cgam^{2}\intM{\nablanorm{2}{A}{2}\gamma^{s-2}}\\
	+\intM{\nablanorm{2}{A}{2}\norm{A}^{2}\gamma^{s}}\Bigr)+c\intM{\nablanorm{2}{A}{2}\gamma^{4}}\intM{\norm{\nabla A}^{2}\gamma^{2}}.\label{EvolutionLemma4,2}
\end{multline}
The last term in $\oo{\ref{EvolutionLemma4,2}}$ can be dealt with via integration by parts, the Cauchy-Schwarz inequality and the interpolation inequality from Lemma $\ref{EvolutionLemma3}$:
\begin{align*}
&\intM{\nablanorm{2}{A}{2}\gamma^{4}}\intM{\norm{\nabla A}^{2}\gamma^{2}}\\
&\leq c\oo{\intM{\norm{\nabla_{\oo{3}}A}\norm{\nabla A}\gamma^{4}}+\cgam\intM{\norm{\nabla_{\oo{2}}A}\norm{\nabla A}\gamma^{3} }}
\oo{\intM{\norm{\nabla_{\oo{2}}A}\norm{A}\gamma^{2}}+\cgam\intM{\norm{\nabla A}\norm{A}\gamma}}\\
&\leq\frac{1}{2}\intM{\nablanorm{2}{A}{2}\gamma^{4}}\intM{\norm{\nabla A}^{2}\gamma^{2}}
+c\A{2}\oo{\intM{\nablanorm{3}{A}{2}\gamma^{6}}+\cgam^{6}\A{2}}.
\end{align*}
Hence
\begin{equation*}
\intM{\nablanorm{2}{A}{2}\gamma^{4}}\intM{\norm{\nabla A}^{2}\gamma^{2}}
\leq c\A{2}\oo{\intM{\nablanorm{3}{A}{2}\gamma^{6}}+\cgam^{6}\A{2}}.
\end{equation*}
Substituting this result into $\oo{\ref{EvolutionLemma4,2}}$ and combining with $\oo{\ref{EvolutionLemma4,1}}$ then yields
\begin{align*}
&\intM{\nablanorm{2}{A}{2}\norm{A}^{2}\gamma^{6}}+\intM{\norm{\nabla A}^{2}\norm{A}^{4}\gamma^{6}}
\\ \ \ &\leq c\A{2}\Bigl(\intM{\nablanorm{3}{A}{2}\gamma^{6}}+\intM{\nablanorm{2}{A}{2}\norm{A}^{2}\gamma^{6}}
+\intM{\norm{\nabla A}^{2}\norm{A}^{4}\gamma^{6}}+\cgam^{6}\A{2}\Bigr),
\end{align*}
so that absorbing and multiplying out yields the statement of the lemma.
\end{proof}

\end{section}

\begin{section}{A-priori estimates for $\vn{\nabla_{(k)}A}^2_{2,\gamma^s}$ along the flow}

The structure of this argument is similar to that used for \cite[Proposition
4.6]{Kuwert2}, although the details are necessarily different.
We begin by applying the evolution equations and interpolation inequalities
above to obtain local control on the concentration of curvature $\vn{A}^2_{2,\gamma^s}$.

\begin{proposition}\label{EvolutionProposition3}
Let $f:\Sigma\times\cc{0,T^{*}}\rightarrow\mathbb{R}^{3}$ satisfy
$\eqref{trilaplacianIntro1}$, and let $\gamma$ be a cutoff function
satisfying $\oo{\ref{E:gammaprop}}$. Then there is an absolute constant
$\varepsilon_{0}>0$ such that if
\[
\sup_{\cc{0,T^{*} }}\int_{\cc{\gamma>0}}{\norm{A}^{2}\,d\mu}\leq\varepsilon_{0}
\]
then for any $t\in\cc{0,T^{*}}$ we have
\begin{align*}
\int_{\cc{\gamma=1}}{\norm{A}^{2}\,d\mu}+\int_{0}^{t}{\int_{\cc{\gamma=1}}{\oo{\nablanorm{3}{A}{2}+\nablanorm{2}{A}{2}\norm{A}^{2}+\norm{\nabla A}^{2}\norm{A}^{4}}\,d\mu}\,d\tau}
\leq\int_{\cc{\gamma>0}}{\norm{A}^{2}\,d\mu}\Big|_{t=0}+c \, \cgam^{6}\, \varepsilon_{0}\, t.
\end{align*}
\end{proposition}
\begin{proof}
Utilising Proposition $\ref{EvolutionProposition1}$ with $s=6,k=0$, yields
\begin{equation}
\frac{d}{dt}\intM{\norm{A}^{2}\gamma^{6}}+\oo{2-\delta}\intM{\nablanorm{3}{A}{2}\gamma^{6}}
\leq\intM{\cc{P_{3}^{4}\oo{A}\star A}\gamma^{6}}+c\, \cgam^{6}\A{2}.\label{EvolutionProposition3,1}
\end{equation}
We estimate the $P$-style terms via integration by parts and the Cauchy-Schwarz inequality, whilst utilising our inequalities from Lemma $\ref{EvolutionLemma4}$:
\begin{align}
\intM{\cc{P_{3}^{4}\oo{A}\star A}\gamma^{6}}
 &=\intM{\oo{\nabla_{\oo{4}}A\star A\star A\star A}\gamma^{6}}\notag\\
 &\qquad +\intM{\cc{\oo{\nabla_{\oo{3}}A\star \nabla A\star
A+\nabla_{\oo{2}}A\star\nabla_{\oo{2}}A\star A+\nabla_{\oo{2}}A\star\nabla
A\star\nabla A}\star A}\gamma^{6}}\nonumber\\
&\leq\intM{\oo{\nabla_{\oo{3}}A\star\nabla A\star A\star A}\gamma^{6}}+\intM{\oo{\nabla_{\oo{3}}A\star A\star A\star A\star\nabla\gamma}\gamma^{5}}\nonumber\\
&\qquad+\delta\intM{\nablanorm{3}{A}{2}\gamma^{6}}+c_{\delta}\intM{\oo{\nablanorm{2}{A}{2}\norm{A}^{2}+\norm{\nabla A}^{2}\norm{A}^{4}+\norm{\nabla A}^{4}}\gamma^{6}}\nonumber\\
&\leq\delta\intM{\nablanorm{3}{A}{2}\gamma^{6}}+c_{\delta}\intM{\oo{\nablanorm{2}{A}{2}\norm{A}^{2}+\norm{\nabla A}^{2}\norm{A}^{4}+\norm{\nabla A}^{4}}\gamma^{6}}\nonumber\\
&\qquad+c \, \cgam^{2}\intM{\norm{A}^{6}\gamma^{4}}.\label{EvolutionProposition3,2}
\end{align}
The last term in $\oo{\ref{EvolutionProposition3,2}}$ can be estimated via
Theorem \ref{MichaelSimon}, the Cauchy-Schwarz inequality and Lemma
$\ref{EvolutionLemma4}$:
\begin{align*}
\cgam^{2}\intM{\norm{A}^{6}\gamma^{4}}&\leq c \, \cgam^{2}\oo{\intM{\norm{\nabla A}\norm{A}\gamma^{2}}+\cgam\intM{\norm{A}^{3}\gamma}+\intM{\norm{A}^{4}\gamma^{2} }}^{2}\\
&\leq c\A{2}\Bigl(\intM{\nablanorm{3}{A}{2}\gamma^{6}}+\intM{\norm{\nabla A}^{2}\norm{A}^{4}\gamma^{6}}
+\cgam^{2}\intM{\norm{A}^{6}\gamma^{4}}+\cgam^{6}\A{2}\Bigr).
\end{align*}
Hence for $\varepsilon_{0}$ sufficiently small, Lemma $\ref{EvolutionLemma4}$ tells us that
\[
\cgam^{2}\intM{\norm{A}^{6}\gamma^{4}}\leq c\int_{\cc{\gamma>0}}{\norm{A}^{2}\,d\mu}\oo{\intM{\nablanorm{3}{A}{2}\gamma^{6}}+\cgam^{6}\int_{\cc{\gamma>0}}{\norm{A}^{2}\,d\mu}}.
\]
Combining this with $\oo{\ref{EvolutionProposition3,2}}$ and substituting into $\oo{\ref{EvolutionProposition3,1}}$ then yields
\begin{equation}
\frac{d}{dt}\intM{\norm{A}^{2}\gamma^{6}}+\frac{3}{2}\intM{\nablanorm{3}{A}{2}\gamma^{6}}\nonumber
\leq c\intM{\oo{\nablanorm{2}{A}{2}\norm{A}^{2}+\norm{\nabla A}^{2}\norm{A}^{4}}\gamma^{6}}+c \, \cgam^{6}\int_{\cc{\gamma>0}}{\norm{A}^{2}\,d\mu}.\label{EvolutionProposition3,3}
\end{equation}
Using Lemma $\ref{EvolutionLemma4}$ and allowing $\varepsilon_{0}$ to be sufficiently small we then obtain
\begin{equation*}
\frac{d}{dt}\intM{\norm{A}^{2}\gamma^{6}}+\intM{\oo{\nablanorm{3}{A}{2}+\nablanorm{2}{A}{2}\norm{A}^{2}+\norm{\nabla A}^{2}\norm{A}^{4}}\gamma^{6}}
\leq c \, \cgam^{6}\int_{\cc{\gamma>0}}{\norm{A}^{2}\,d\mu}.\nonumber
\end{equation*}
Integrating over $\cc{0,t}$ and using $\cc{\gamma=1}\subset\cc{\gamma>0}$, $0\leq\gamma\leq1$ completes the proof.
\end{proof}
We now control the evolution of $\vn{\nabla_{(k)}A}^2_{2,\gamma^s}$ for $k\ge1$.
Our goal is to eventually apply the interpolation inequalities to the resultant
estimate in order to prove that $\vn{\nabla_{(k)}A}^2_{2,\gamma^s}$ satisfies a
differential inequality and is therefore bounded via Gronwall's inequality.
\begin{proposition}\label{EvolutionProposition5}
Let $f:\Sigma\times\left[0,T\right)\rightarrow\mathbb{R}^{3}$ satisfy \eqref{trilaplacianIntro1}
and let $\gamma$ be a cutoff function satisfying $\oo{\ref{E:gammaprop}}$. Then
for $s\geq2k+6$ the following estimate holds:
\begin{multline*}
\frac{d}{dt}\intM{\nablanorm{k}{A}{2}\gamma^{s}}+\intM{\nablanorm{k+3}{A}{2}\gamma^{s}}\\
  \leq c\llll{A}_{\infty,\cc{\gamma>0}}^{6}\intM{\nablanorm{k}{A}{2}\gamma^{s}}+c \, \cgam^{2k}\int_{\cc{\gamma>0}}{\norm{A}^{2}\,d\mu}\oo{\cgam^{6}+\llll{A}_{\infty,\cc{\gamma>0}}^{6}}.
\end{multline*}
\end{proposition}
\begin{proof}
We use integration by parts, Proposition $\ref{EvolutionProposition4}$ and the Cauchy-Schwarz inequality to deal with the $P$-style terms that appear in Proposition $\ref{EvolutionProposition1}$:
\begin{align}
&\intM{ \left( P_{3}^{k+4}\oo{A}\star\nabla_{\oo{k}}A \right) \gamma^{s}}\nonumber\\
&=\intM{\cc{\oo{\nabla_{\oo{k+4}}A\star A\star A+\nabla_{\oo{k+3}}A\star\nabla A\star A}\star\nabla_{\oo{k}}A}\gamma^{s}}
+\intM{P_{4}^{2k+4,k+2}\oo{A} \cdot \gamma^{s}}\nonumber\\
&\leq\intM{\cc{\nabla_{\oo{k+3}}A\star\oo{\nabla_{\oo{k}}A\star\nabla A\star A+\nabla_{\oo{k+1}}A\star A\star A}\gamma^{s} }}\nonumber\\
&\quad +c \, \cgam\intM{\norm{\nabla_{\oo{k+3}}A}\norm{\nabla_{\oo{k}}A}\norm{A}^{2}\gamma^{s-1}}+\intM{P_{4}^{2k+4,k+2}\oo{A}\gamma^{s}}\nonumber\\
&\leq\delta\intM{\nablanorm{3}{A}{2}\gamma^{s}}+\intM{\left( P_{4}^{2k+4,k+2}\oo{A}+P_{6}^{2k+2,k+1}\oo{A} \right) \gamma^{s}}
+\cgam^{2}\intM{P_{6}^{2k,k}\oo{A}\gamma^{s-2}}\nonumber\\
&\leq\delta\intM{\nablanorm{3}{A}{2}\gamma^{s}}+c\llll{A}_{\infty,\cc{\gamma>0}}^{2}\Bigl(\intM{\nablanorm{k+2}{A}{2}\gamma^{s}}+\cgam^{2\oo{k+2}}\A{2}\Bigr)\nonumber\\
&\quad+c\llll{A}_{\infty,\cc{\gamma>0}}^{4}\oo{\intM{\nablanorm{k+1}{A}{2}\gamma^{s}}+\cgam^{2\oo{k+1}}\A{2}}\nonumber\\
&\quad +c \, \cgam^{2}\llll{A}_{\infty,\cc{\gamma>0}}^{4}\oo{\intM{\nablanorm{k}{A}{2}\gamma^{s-2}}+\cgam^{2k}\A{2}}.\label{EvolutionProposition1,1}
\end{align}
Integrating by parts, applying the Cauchy-Schwarz inequality and absorbing and multiplying out similarly as in the proof of Lemma $\ref{EvolutionLemma3}$ we obtain the following interpolation inequalities:
\begin{multline}
\llll{A}_{\infty,\cc{\gamma>0}}^{2}\intM{\nablanorm{k+2}{A}{2}\gamma^{s}}\leq \delta\intM{\nablanorm{k+3}{A}{2}\gamma^{s}}\\
+c\llll{A}_{\infty,\cc{\gamma>0}}^{4}\intM{\nablanorm{k+1}{A}{2}\gamma^{s}}
 +c \, \cgam^{2k}\A{2}\oo{\cgam^{6}+\llll{A}_{\infty,\cc{\gamma>0}}^{6}},\label{EvolutionProposition1,2}
\end{multline}
\begin{multline}
\llll{A}_{\infty,\cc{\gamma>0}}^{4}\intM{\nablanorm{k+1}{A}{2}\gamma^{s}}\leq \delta\llll{A}_{\infty,\cc{\gamma>0}}^{2}\intM{\nablanorm{k+2}{A}{2}\gamma^{s}}\\
+c\llll{A}_{\infty,\cc{\gamma>0}}^{6}\intM{\nablanorm{k}{A}{2}\gamma^{s}} +c \, \cgam^{2k}\A{2}\oo{\cgam^{6}+\llll{A}_{\infty,\cc{\gamma>0}}^{6}},\label{EvolutionProposition1,3}
\end{multline}
and
\begin{equation}
\cgam^{2}\llll{A}_{\infty,\cc{\gamma>0}}^{4} \leq \delta\intM{\nablanorm{k+3}{A}{2}\gamma^{s}}
+c\llll{A}_{\infty,\cc{\gamma>0}}^{6}\intM{\nablanorm{k}{A}{2}\gamma^{s}}+c \, \cgam^{2\oo{k+3}}\A{2}.\label{EvolutionProposition1,4}
\end{equation}
Combining $\oo{\ref{EvolutionProposition1,2}},\oo{\ref{EvolutionProposition1,3}}$ and $\oo{\ref{EvolutionProposition1,4}}$ and substituting into $\oo{\ref{EvolutionProposition1,1}}$ estimates of all the $P$-style terms. Substituting this result into Proposition $\ref{EvolutionProposition5}$ then yields the desired result.
\end{proof}

We are now in a position to prove an analogue of \cite[Proposition 4.6]{Kuwert2}.

\begin{proposition}\label{EvolutionProposition6}
Let $f:\Sigma\times\cc{0,T^{*}}\rightarrow\mathbb{R}^{3}$ satisfy
\eqref{trilaplacianIntro1} and let $\gamma$ be a cutoff function satisfying
$\oo{\ref{E:gammaprop}}$. Then there is a $\varepsilon_{0}>0$ such that if
\begin{equation}
\sup_{\cc{0,T^{*} }}\int_{\cc{\gamma>0}}{\norm{A}^{2}\,d\mu}\leq\varepsilon_{0},\label{EvolutionProposition6,1}
\end{equation}
we can conclude that for all $k\in \mathbb{N}$ there are constants $\tilde c_k$, depending only on $T^*$, $\cgam$,
$\varepsilon_0$ and
\begin{equation}
\alpha_{0}\oo{k+3}:=\sum_{j=0}^{k+3}\llll{\nabla_{\oo{j}}A}_{2,\cc{\gamma>0}}^{2}\Big|_{t=0}\label{Alphanotation}
\end{equation}
such that 
\begin{equation}
\llll{\nabla_{\oo{k}}A}_{\infty,\cc{\gamma=1}}^{2}\leq \tilde c_k.\label{EvolutionProposition6,2}
\end{equation}
\end{proposition}
\begin{proof}
We introduce a set of cutoff functions $\gamma_{\sigma,\tau}$ for $0\leq\sigma<\tau\leq1$ satisfying $\gamma_{\sigma,\tau}=1$ for $\gamma\geq\tau$ and $\gamma_{\sigma,\tau}=0$ for $\gamma\leq\sigma$. Applying Proposition $\ref{EvolutionProposition3}$ with $\gamma=\gamma_{0,\frac{1}{2}}$  yields
\begin{equation}
\int_{0}^{T^{*}}{\int_{\cc{\gamma\geq\frac{1}{2} }}{\nablanorm{3}{A}{2}+\nablanorm{2}{A}{2}\norm{A}^{2}+\norm{\nabla A}^{2}\norm{A}^{4}\,d\mu}\,d\tau}\nonumber\leq c\varepsilon_{0}\oo{1+\cgam^{6}\, T^{*}}.\label{EvolutionProposition6,3}
\end{equation}
Next, we apply Proposition $\ref{EvolutionProposition2}$ with $T=A$ and $\gamma_{\frac{1}{2},\frac{3}{4}}$. This gives
\begin{equation}
\int_{0}^{T^{*}}{\llll{A}_{\infty,\cc{\gamma\geq\frac{3}{4} }}^{6}\,d\tau}\leq c\varepsilon_{0}^{3}\oo{1+\cgam^{6}\, T^{*}}\label{EvolutionProposition6,4}.\nonumber
\end{equation}
Here we have used $\cc{\gamma\geq\frac{3}{4}}\subseteq\cc{\gamma_{\frac{1}{2},\frac{3}{4}}=1},\cc{\gamma_{\frac{1}{2},\frac{3}{4} }}\subseteq\cc{\gamma\geq\frac{1}{2}}$ and $\oo{\ref{EvolutionProposition6,4}}$. Next we look at Proposition $\ref{EvolutionProposition5}$ with $\gamma_{\frac{3}{4},\frac{7}{8}}$ and integrate over $\cc{0,T^{*}}$. This gives us for any $t \in \cc{0,T^{*}}$:
\begin{equation*}
\intM{\nablanorm{k}{A}{2}\gamma_{\frac{3}{4},\frac{7}{8}}^{s}}\leq \llll{\nabla_{\oo{k}}A}_{2,\cc{\gamma>0}}^{2}\Big|_{t=0}\\
+c \, \cgam^{2k}\varepsilon_{0}\oo{1+\cgam^{6}\, T^{*}}
+c\int_{0}^{t}{\llll{A}_{\infty,\cc{\gamma\geq\frac{3}{4} }}^{6}\intM{\nablanorm{k}{A}{2}\gamma_{\frac{3}{4},\frac{7}{8}}^{s}}\,d\tau}.
\end{equation*}
Applying Gromwall's Inequality whilst taking into account identity
$\oo{\ref{EvolutionProposition6,4}}$ gives
\begin{align*}
\int_{\cc{\gamma\geq\frac{7}{8} }}{\nablanorm{k}{A}{2}\,d\mu}&\leq c\oo{\llll{\nabla_{\oo{k}}A}_{2,\cc{\gamma>0}}^{2}\Big|_{t=0}+c \, \cgam^{2k}\varepsilon_{0}\oo{1+\cgam^{6}\, T^{*} }}e^{c\varepsilon_{0}^{3}\oo{1+\cgam^{6}\, T^{*} }}
\\
&\leq c_{k}\oo{\alpha_{0}\oo{k},\cgam,\varepsilon_{0},T^{*}}\,.
\end{align*}
Combining this with Proposition $\ref{EvolutionProposition2}$ with $\gamma_{\frac{7}{8},\frac{15}{16}}$,
\begin{align*}
\llll{\nabla_{\oo{k}}A}_{\infty,\cc{\gamma\geq\frac{15}{16} }}^{6}
&\leq c \, c_{k}^{2}\Bigl(c_{k+3}+\intM{\cc{P_{4}^{2\oo{k+2},k+2}\oo{A}+P_{6}^{2\oo{k+1},k+1}\oo{A}}\gamma_{\frac{7}{8},\frac{15}{16}}^{s}}+\cgam^{2\oo{k+3}}\varepsilon_{0}\Bigr)\\
&\leq c\oo{c_{3},c_{k},c_{k+2},c_{k+3},\varepsilon_{0},T^{*},\cgam},
\end{align*}
which proves $\oo{\ref{EvolutionProposition6,2}}$. Here we have used the inequality
\[
\llll{A}_{\infty,\cc{\gamma=1}}^{6}\leq c\, \varepsilon_{0}^{2}\oo{c_{3}+\cgam^{6}\varepsilon_{0}}
\]
which can be obtained by combining Proposition $\ref{EvolutionProposition2}$ and Lemma $\ref{EvolutionLemma4}$.
\end{proof}

We now have sufficient tools to prove Theorem \ref{LifespanTheorem}. However,
we will first sharpen our derivative bounds from Proposition
$\ref{EvolutionProposition6}$ for later application in combination with Theorem
$\ref{BlowupTheorem1}$ where we construct a blowup.
Theorem \ref{EvolutionProposition6} will also be used to obtain our preliminary
compactness result (Lemma \ref{ConvergenceLemma1}) for the flow once global
existence has been obtained.

\begin{theorem}[Interior Estimates]\label{InteriorEstimatesTheorem1}
Suppose $f:\Sigma\times\left(0,T^{*}\right]\rightarrow\mathbb{R}^{3}$ satisfies \eqref{trilaplacianIntro1} and
\begin{equation}
\sup_{0<t\leq T^{*}}\int_{f^{-1}\oo{B_{\rho}\oo{x} }}{\norm{A}^{2}\,d\mu}\leq\varepsilon_{0}\text{  for  }T^{*}\leq c\rho^{6}.\label{InteriorEstimatesTheorem1,1}
\end{equation}
Then for any $m\in\mathbb{N}_{0}$ we have at time $t\in\left(0,T^{*}\right]$ we have the estimates
\begin{equation*}
\llll{\nabla_{\oo{m}}A}_{2,f^{-1}(B_{\frac{\rho}{2}}\oo{x})} \leq c_{m}\sqrt{\varepsilon_{0}}t^{-\frac{m}{6}} \mbox{ and }
\llll{\nabla_{\oo{m}}A}_{\infty,f^{-1}(B_{\frac{\rho}{2}}\oo{x})} \leq c_{k}\sqrt{\varepsilon_{0}}t^{-\frac{m+1}{6}},
\end{equation*}
where the constants $c_{m}$ depend only upon $m$, $\rho$, $T^{*}$ and $\norm{\nabla_{\oo{m}}A}_{2,f^{-1}\oo{B_{\rho}\oo{x_{0} }} }\Big|_{t=0}$.
\end{theorem}
The idea of the proof is to consider a family of cutoff functions in time and
integrate over $\left(0,T^{*}\right]$. An inductive argument then gives the
desired the result. Since this is similar to \cite[Theorem 3.5]{Kuwert1}, we
have omitted it.

\end{section}

\begin{section}{Proof of the lifespan theorem}

Now that we have pointwise bounds on $\vn{\nabla_{(k)}A}^2_{2,\gamma^s}$ when
the concentration of curvature is small, we are able to proceed with our proof
of Theorem \ref{LifespanTheorem}.
The proof will rest upon contradicting the maximal time $T$ of smooth
existence. We may assume $\rho=1$ in assumption $\oo{\ref{Lifespansmallness}}$
due to scale invariance of the $\vn{A}_2^2$.
That is, under the rescaling
$\tilde{f}\oo{p,t}=f\oo{\frac{p}{\rho},\frac{t}{\rho^{6} }}$ we have
\[
\int_{f^{-1}\oo{B_{\rho} }}{\norm{A}^{2}\,d\mu}=\int_{\tilde{f}^{-1}\oo{B_{1} }}{\tilde{\norm{A}}^{2}\,d\tilde{\mu}}.\nonumber
\]
Under such a rescaling, we aim to show that
\[
T\geq\frac{1}{C}
\]
and that for $0\leq t\leq\frac{1}{C}$ the estimate
\[
\int_{f^{-1}\oo{B_{1} }}{\norm{A}^{2}\,d\mu}\leq C\, \varepsilon_{0}
\]
holds. We make the definition
\begin{equation}
\Kappa\oo{t}=\sup_{x\in\mathbb{R}^{3}}\llll{A}_{2,f^{-1}\oo{B_{1}\oo{x} }}^{2}.\label{LifespanProof1}
\end{equation}
We can cover the ball $B_{1}$ with a number of translated copies of $B_{\frac{1}{2}}$. It follows that there is an absolute constant $C_{\Kappa}$ such that
\begin{equation}
\Kappa\oo{t}\leq C_{\Kappa}\sup_{x\in\mathbb{R}^{3}}\llll{A}_{2,f^{-1}(B_{\frac{1}{2}}\oo{x})}^{2}.\label{LifespanProof2}
\end{equation}
Now because $f$ is smooth by the definition of our family of compact immersions $\Sigma_{t}$ we have $f\oo{\Sigma\times\cc{0,t}}$ compact for $t<T$, and so $\Kappa\in C\oo{\left[0,T\right)}$. Firstly we define $C_{0}$ to be the constant on the right hand side of Proposition $\ref{EvolutionProposition3}$. From this we make the definition $\lambda=\frac{1}{C_{0}\cgam^{6}}$, and 
\begin{equation}
t_{0}=\sup\left\{0\leq t\leq\min\left\{T,\lambda\right\}:\Kappa\oo{\tau}\leq3\, C_{\Kappa}\, \varepsilon_{0}\text{  for  }0\leq\tau\leq t\right\}.\label{LifespanProof3}
\end{equation}
The reason for this parameter $\lambda$ will become apparent later when we establish a contradiction.

We will go through this proof in $3$ steps, labelled $\oo{\ref{LifespanProof4}}-\oo{\ref{LifespanProof6}}$. These are as follows:
\begin{equation}
\text{ Show that }t_{0}=\text{min}\left\{T,\lambda\right\},\label{LifespanProof4}
\end{equation}
\begin{equation}
\text{Show that if }t_{0}=\lambda\text{ then we have the Lifespan Theorem},\label{LifespanProof5}
\end{equation}
and 
\begin{equation}
\text{Show that if }T\neq\infty\text{ then }t_{0}\neq T.\label{LifespanProof6}
\end{equation}
Note that $\oo{\ref{LifespanProof6}}$ is equivalent to showing that $t_{0}=T\implies T=\infty$. We claim that the three statements $\oo{\ref{LifespanProof4}},\oo{\ref{LifespanProof5}}$ and $\oo{\ref{LifespanProof6}}$ together prove the Lifespan Theorem. To see this, note that if $\oo{\ref{LifespanProof6}}$ holds then we must have either $t_{0}=T$ or $t_{0}=\lambda$. If $t_{0}=T$, then $\oo{\ref{LifespanProof6}}$ implies that $T=\infty,$ meaning that the flow exists for all time, which proves the Lifespan Theorem. If instead $t_{0}=\lambda$, then $\oo{\ref{LifespanProof5}}$ will directly give us the Lifespan Theorem. 

To begin, we note that by one of the assumptions of the theorem we have
\[
\Kappa\oo{0}=\sup_{x\in\mathbb{R}^{3}}\llll{A}_{2,f^{-1}\oo{B_{1}\oo{x} }}^{2}\Big|_{t=0}\leq\varepsilon_{0}< 3\, C_{\Kappa}\, \varepsilon_{0},
\]
and so we must have $t_{0}$ strictly positive. To see this, we note that the
continuity of $\Kappa$ and the fact that
$\Kappa\oo{0}<3\,C_{\Kappa}\,\varepsilon_{0}$ forces
$\Kappa\oo{t}<3\,C_{\Kappa}\,\varepsilon_{0}$ for some strictly positive time
period. The definition of $t_{0}$ then guarantees $t_{0}>0$.
Also, by the definition of $t_{0}$ and by continuity of $\Kappa,$ we must have 
\begin{equation}
\Kappa\oo{t_{0}}=3\, C_{\Kappa}\, \varepsilon_{0}.\label{LifespanProof7}
\end{equation}
Let us assume that $t_{0}<\text{min}\left\{\lambda, T\right\}$ and aim for a contradiction. We choose a cutoff function $\gamma$ such that
\[
\chi_{B_{\frac{1}{2}}\oo{x}}\leq\tilde{\gamma}\leq\chi_{B_{1}\oo{x}}\text{  for any }x\in\Sigma_{t}.
\]
By the initial smoothness of $\norm{A}^{2}$, we can always find a $\rho^{\star}>0$ in assumption $\oo{\ref{Lifespansmallness}}$ that is small enough to guarantee the smallness condition $\oo{\ref{EvolutionProposition6,1}}$ 
holds for every $T^{*}<t_{0}$. Note that this is equivalent to assuming the hypothesis of Proposition $\ref{EvolutionProposition3}$ on the interval $\left[0,t_{0}\right)$.
The results of Proposition $\ref{EvolutionProposition3}$ then imply that
\[
\llll{A}_{2,\cc{\gamma=1}}^{2}\leq\llll{A}_{2,\cc{\gamma>0}}^{2}\Big|_{t=0}+C_{0}\, \cgam^{6}\llll{A}_{2,\cc{\gamma>0}}^{2}t,
\]
from which we conclude
\begin{equation*}
\int_{f^{-1}(B_{\frac{1}{2}}\oo{x})}{\norm{A}^{2}\,d\mu}
\leq\int_{f^{-1}\oo{B_{1}\oo{x} }}{\norm{A}^{2}\,d\mu}\Bigg|_{t=0}
+C_{0}\,\cgam^{6}\int_{f^{-1}\oo{B_{1}\oo{x} }}{\norm{A}^{2}\,d\mu}
\leq\varepsilon_{0}+C_{0}\, \cgam^{6}\, \varepsilon_{0}\, t_{0}
\end{equation*}
for $t\in\left[0,t_{0}\right)$ and $x\in\mathbb{R}^{3}$.
It follows from the assumption that $t_{0}<\lambda$ and the definition of $\lambda$ that for $t\in\left[0,t_{0}\right)$ and $x\in\mathbb{R}^{3}$ that we have
\[
\int_{f^{-1}(B_{\frac{1}{2}}\oo{x})}{\norm{A}^{2}d\mu}\leq\varepsilon_{0}+C_{0}\, \cgam^{6}\, \varepsilon_{0}\, t_{0}<2\, \varepsilon_{0}.
\]
We deduce from $\oo{\ref{LifespanProof2}}$ that
\[
\Kappa\oo{t}\leq C_{\Kappa}\sup_{x\in\mathbb{R}^{3}}\norm{A}_{2,f^{-1}(B_{\frac{1}{2}}\oo{x})}^{2}\leq 2\, C_{\Kappa}\, \varepsilon_{0}\text{  for  }t\in\left[0,t_{0}\right).
\]
By the continuity of $\Kappa$ this means that 
\[
\lim_{t\nearrow t_{0}}\Kappa\oo{t}\leq 2\,C_{\Kappa}\,\varepsilon_{0}\,.
\]
This contradicts statement $\oo{\ref{LifespanProof7}}$ where we said that $\Kappa\oo{t_{0}}=3\, C_{\Kappa}\, t_{0}$. Hence $\oo{\ref{LifespanProof4}}$ holds.

We now note that under assumption $\oo{\ref{LifespanProof4}}$ (which was just proven) we must have either $t_{0}=T$ or $t_{0}=\lambda$. If $t_{0}=\lambda$ then because $t_{0}=\min\left\{\lambda, T\right\}$ it follows that
\[
T\geq t_{0}=\lambda=\frac{1}{C_{0}\,\cgam^{6}},
\]
and the definition of $t_{0}$ tells us that
\[
\int_{f^{-1}\oo{B_{1}\oo{x} }}{\norm{A}^{2}\,d\mu}\leq 3\, C_{\Kappa}\, \varepsilon_{0}
\]
for $x\in\Sigma_{t}$ and $t\in\cc{0,\lambda}$. Together these two statements
give us both statements $\oo{\ref{lifespan1}}$ and $\oo{\ref{lifespan2}}$ of
the theorem with $C=\max\left\{C_{0}\, \cgam^{6},3\, C_{\Kappa}\right\}$. That
is, assuming that $t_{0}=\lambda$, we have the Lifespan Theorem. This is
$\oo{\ref{LifespanProof5}}$.

Finally, we turn our attention to $\oo{\ref{LifespanProof6}}$. We will assume
\[
t_{0}=T\neq\infty
\] 
and then aim to contradict the maximality of $T$. We note that we have not included the case $T=\infty$ because in that case the second part of the Lifespan Theorem $\oo{\ref{lifespan1}}$ would hold automatically. Additionally, we would have $\lambda\leq T$ so that $t_{0}=\lambda=T$ and then statement $\oo{\ref{lifespan2}}$ with $C=2\, C_{\Kappa}$ would directly follow from our earlier estimate:
\begin{equation*}
\Kappa\oo{t}\leq 2\, C_{\Kappa}\, \varepsilon_{0},t\in\left[0,t_{0}\right).
\end{equation*}
We can also exclude the case $T<\lambda$ for the sake of our argument, because
by $\oo{\ref{LifespanProof4}}$ it would follow that $t_{0}=T$, from which we
could conclude (using $\oo{\ref{LifespanProof5}}$) the Lifespan Theorem. Next,
Proposition $\ref{EvolutionProposition6}$ allows us to establish a uniform
bound on all derivatives of $A$ up to and including the final time $T$. A
standard argument similar to that of Hamilton \cite{Hamilton1} allows us to
conclude that $\Sigma_{t}\rightarrow\Sigma_{T}$ in the $C^{\infty}$ topology
and that each $f(\cdot,t)$ induces an equivalent family of metrics, implying
the uniqueness (up to reparametrisation) of our limiting object $\Sigma_{T}$ up
to rigid motions.
Hence we may define a new family of immersions
$h:\Sigma\times\left[0,\delta\right)\rightarrow\mathbb{R}^{3}$ satisfying
\eqref{trilaplacianIntro1}, given by $h\oo{\cdot,t}=f\oo{\cdot,t+T}$, so that
$h$ has smooth initial data $f\oo{\cdot,T}$.
Theorem \ref{TMste} implies existence of the flow $h$ for some positive time
$\delta>0$, extending the smooth existence of our original flow to the interval
$\cc{0,T+\delta}\supset\cc{0,T}$, contradicting the maximality of $T$.
Hence we have established $\oo{\ref{LifespanProof6}}$.
This means we have proven steps
$\oo{\ref{LifespanProof4}},\oo{\ref{LifespanProof5}}$ and
$\oo{\ref{LifespanProof6}}$, which as noted earlier completes the proof of Theorem \ref{LifespanTheorem}.
\end{section}

\begin{section}{A pointwise estimate for the tracefree curvature}

In this section we present some new uses of geometric quantities involving iterated covariant derivatives of $A$, $A^o$ and $H$.
These will include local integral estimates and a multiplicative Sobolev-type inequality, aimed at the establishemnt of a pointwise bound
for the tracefree curvature $A^{o}$ that only depends on $\llll{\nabla\Delta
H}_{2,\cc{\gamma>0}}^{2}$, $\llll{A^{o}}_{2,\cc{\gamma>0}}^{2}$, and the
gradient bound for $\gamma$.

This holds for any immersed surface with sufficient weak regularity, and so is of independent interest.
For our purposes in this paper the estimate will also be useful in the proof of Theorem \ref{GapTheorem1} in the next section.

We begin with Simons' identity-type relations.

\begin{lemma}\label{TracefreeLemma1}
The following identities hold:
\begin{equation}
\Delta A^{o}=S^{o}\oo{\nabla_{\oo{2}}H}+\frac{1}{2}H^{2}A^{o}-\norm{A^{o}}^{2}A^{o}\label{TracefreeLemma1,1}
\end{equation}
and
\begin{equation}
\Delta\nabla H=\nabla\Delta H+\frac{1}{4}\nabla H\norm{H}^{2}+\nabla H\star\oo{HA^{o}+A^{o}\star A^{o}}.\label{TracefreeLemma1,2}
\end{equation}
Consequently, there is a universal constant $c$ such that
\begin{equation}
-\inner{\Delta\nabla H,\nabla H}\leq-\inner{\nabla\Delta H,\nabla H}-\frac{1}{8}\norm{\nabla H}^{2}H^{2}+c\norm{\nabla H}^{2}\norm{A^{o}}^{2}\,.\label{TracefreeLemma1,3}
\end{equation}
\end{lemma}
\begin{proof}
The first identity is identical to one in the Appendix of \cite{Helfrich1}. For the second identity, we employ the Riemann curvature tensor:
\begin{equation*}
\nabla_{ijk}H = \nabla_{ikj}H=\nabla_{kij}H+R_{ikjs}g^{st}\nabla_{t}H
 = \nabla_{kij}H+A_{k}^{t}A_{ij}\nabla_{t}H-A_{i}^{t}A_{kj}\nabla_{t}H.
\end{equation*}
This immediately gives
\begin{equation}
\Delta\nabla_{k}H=\nabla_{k}\Delta H+HA_{k}^{t}\nabla_{t}H-A_{i}^{t}A_{k}^{i}\nabla_{t}H.\label{TracefreeLemma1,5}
\end{equation}
Evaluating the second term on the right, we have
\[
HA_{k}^{t}\nabla_{t}H=H\oo{\oo{A^{o}}_{k}^{t}+\frac{1}{2}\delta_{k}^{t}H}\nabla_{t}H=\frac{1}{2}\nabla_{k}H\norm{H}^{2}+H\nabla H\star A^{o}.
\]
The last term in $\oo{\ref{TracefreeLemma1,5}}$ is dealt with similarly:
\begin{equation*}
A_{i}^{t}A_{k}^{i}\nabla_{t}H=\oo{\oo{A^{o}}_{i}^{t}+\frac{1}{2}\delta_{i}^{t}H}\oo{\oo{A^{o}}_{k}^{i}+\frac{1}{2}\delta_{k}^{i}H}\nabla_{t}H
 =\frac{1}{4}\nabla_{k}H\norm{H}^{2}+\nabla H\star\oo{HA^{o}+A^{o}\star A^{o}}.
\end{equation*}
Substituting the two preceding results into $\oo{\ref{TracefreeLemma1,5}}$ then
yields the second statement of the lemma. For the last statement, we take the
results of $\oo{\ref{TracefreeLemma1,5}}$ and apply the Peter-Paul inequality
with $p=8,q=\frac{8}{7}$:
\begin{align*}
-\inner{\Delta\nabla H,\nabla H}&\leq-\inner{\nabla\Delta H,\nabla H}-\frac{1}{4}\norm{\nabla H}^{2}H^{2}+c\norm{\nabla H}^{2}\oo{\norm{A^{o}}\norm{H}+\norm{A^{o}}^{2}}\\
&\leq-\inner{\nabla\Delta H,\nabla H}-\frac{1}{8}\norm{\nabla H}^{2}H^{2}+c\norm{\nabla H}^{2}\norm{A^{o}}^{2}.
\end{align*}
\end{proof}

We now prove a multiplicative Sobolev inequality.

\begin{lemma}\label{Tracefreelemma2}
If $f:\Sigma\rightarrow\mathbb{R}^{3}$ is an immersion satisfying 
\begin{equation}
\int_{\cc{\gamma>0}}{\norm{A^{o}}^{2}\,d\mu}\leq\varepsilon_{0}\label{Tracefreesmallness}
\end{equation}
for $\varepsilon_{0}>0$ sufficiently small, then for a universal constant $c>0$,
\begin{equation}
\intM{\norm{\nabla A^{o}}^{2}\gamma^{2}}+\intM{\norm{A^{o}}^{2}H^{2}\gamma^{2}}\\
\leq c\Ao{\frac{4}{3}}\oo{\intM{\norm{\nabla\Delta H}^{2}\gamma^{6} }}^{\frac{1}{3}}+c\,\cgam^{2}\Ao{2},
\label{Tracefreelemma2,1}
\end{equation}
\begin{equation}
\intM{\nablanorm{2}{H}{2}\gamma^{4}}+\intM{\norm{\nabla H}^{2}H^{2}\gamma^{4}}\\
 \leq c\Ao{\frac{2}{3}}\oo{\intM{\norm{\nabla\Delta H}^{2}\gamma^{6} }}^{\frac{2}{3}}+c\, \cgam^{4}\Ao{2},\label{Tracefreelemma2,2}
\end{equation}
and
\begin{equation}
\intM{\norm{A^{o}}^{4}\gamma^{2}}
 \leq c\Ao{\frac{10}{3}}\oo{\intM{\norm{\nabla\Delta H}^{2}\gamma^{6} }}^{\frac{1}{3}}+c\, \cgam^{2}\Ao{4}.\label{Tracefreelemma2,3}
\end{equation}
\end{lemma}
\begin{proof}
For the first inequality we combine integration by parts, identity $\oo{\ref{TracefreeLemma1,1}}$ and Theorem \ref{MichaelSimon} with $u=\norm{A^{o}}^{2}\gamma$:
\begin{align*}
\intM{\norm{\nabla A^{o}}^{2}\gamma^{2}}
&\leq-\intM{\inner{A^{o},\Delta A^{o}}\gamma^{2}}+2\, \cgam\intM{\norm{\nabla A^{o}}\norm{A^{o}}\gamma}\\
&\leq\llll{A^{o}}_{2,\cc{\gamma>0}}\oo{\intM{\nablanorm{2}{H}{2}\gamma^{4} }}^{\frac{1}{2}}-\frac{1}{2}\intM{\norm{A^{o}}^{2}H^{2}\gamma^{2}}\\
&\quad+\cmss\oo{2\intM{\norm{\nabla A^{o}}\norm{A^{o}}\gamma}+\cgam\Ao{2}+\intM{\norm{A^{o}}^{2}\norm{H}\gamma}}^{2}\\
&\quad+\frac{1}{2}\intM{\norm{\nabla A^{o}}^{2}\gamma^{2}}+2\, \cgam^{2}\Ao{2}\\
&\leq\llll{A^{o}}_{2,\cc{\gamma>0}}\oo{\intM{\nablanorm{2}{H}{2}\gamma^{4} }}^{\frac{1}{2}}+\oo{c\Ao{2}-\frac{1}{2}}\intM{\norm{A^{o}}^{2}H^{2}\gamma^{2}}\\
&\quad+\oo{c\Ao{2}+\frac{1}{2}}\intM{\norm{\nabla A^{o}}^{2}\gamma^{2}}+c \, \cgam^{2}\Ao{2}.
\end{align*}
Then for $\varepsilon_{0}>0$ sufficiently small we can absorb and multiply out, yielding
\begin{equation}
\intM{\norm{\nabla A^{o}}^{2}\gamma^{2}}+\intM{\norm{A^{o}}^{2}H^{2}\gamma^{2}}\\
 \leq 2\llll{A^{o}}_{2,\cc{\gamma>0}}\oo{\intM{\nablanorm{2}{H}{2}\gamma^{4} }}^{\frac{1}{2}}+c \, \cgam^{2}\Ao{2}.\label{Tracefreelemma2,4}
\end{equation}
We will leave this identity for the moment. For the second inequality we
combine integration by parts, identity $\oo{\ref{Tracefreelemma2,1}}$, Lemma
\ref{TracefreeLemma1} and Theorem \ref{MichaelSimon} with $u=\norm{\nabla
H}\norm{A^{o}}\gamma^{2}$ to obtain:
\begin{align*}
\intM{\nablanorm{2}{H}{2}\gamma^{4}}
&\leq-\intM{\inner{\nabla H,\Delta\nabla H}\gamma^{4}}+4\, \cgam\intM{\norm{\nabla_{\oo{2}}H}\norm{\nabla H}\gamma^{3}}\\
&\leq2\oo{\intM{\norm{\nabla A^{o}}^{2}\gamma^{2}}\cdot\intM{\norm{\nabla\Delta H}^{2}\gamma^{6} }}^{\frac{1}{2}}-\frac{1}{8}\intM{\norm{\nabla H}^{2}H^{2}\gamma^{4}}\\
&\quad+c\intM{\norm{\nabla H}^{2}\norm{A^{o}}^{2}\gamma^{4}}+\frac{1}{2}\intM{\nablanorm{2}{H}{2}\gamma^{4}}+32\, \cgam^{2}\intM{\norm{\nabla A^{o}}^{2}\gamma^{2}}\\
&\leq2\oo{\intM{\norm{\nabla A^{o}}^{2}\gamma^{2}}\cdot\intM{\norm{\nabla\Delta H}^{2}\gamma^{6} }}^{\frac{1}{2}}-\frac{1}{8}\intM{\norm{\nabla H}^{2}H^{2}\gamma^{4}}\\
&\quad+c\bigg(\intM{\norm{\nabla_{\oo{2}}H}\norm{A^{o}}\gamma^{2}}+\intM{\norm{\nabla H}\norm{\nabla A^{o}}\gamma^{2}}+\cgam\intM{\norm{\nabla H}\norm{A^{o}}\gamma}\\
&\qquad+\intM{\norm{\nabla H}\norm{A^{o}}\norm{H}\gamma^{2}}\bigg)^{2}+\frac{1}{2}\intM{\nablanorm{2}{H}{2}\gamma^{4}}+32\, \cgam^{2}\intM{\norm{\nabla A^{o}}^{2}\gamma^{2}}\\
&\leq2\oo{\intM{\norm{\nabla A^{o}}^{2}\gamma^{2}}\cdot\intM{\norm{\nabla\Delta H}^{2}\gamma^{6} }}^{\frac{1}{2}}+c \, \cgam^{2}\intM{\norm{\nabla A^{o}}^{2}\gamma^{2}}\\
&\quad+\oo{c\Ao{2}+\frac{1}{2}}\intM{\nablanorm{2}{H}{2}\gamma^{4}}+c\oo{\intM{\norm{\nabla A^{o}}^{2}\gamma^{2} }}^{2}\\
&\quad+\oo{c\Ao{2}-\frac{1}{8}}\intM{\norm{\nabla H}^{2}H^{2}\gamma^{4}}.
\end{align*}
Absorbing and multiplying out, whilst combining with $\oo{\ref{Tracefreelemma2,4}}$ gives for $\varepsilon_{0}$ sufficiently small:
\begin{align}
\intM{\nablanorm{2}{H}{2}\gamma^{4}}+\intM{\norm{\nabla H}^{2}H^{2}\gamma^{4}}
&\leq c\oo{\intM{\norm{\nabla A^{o}}^{2}\gamma^{2}}\cdot\intM{\norm{\nabla\Delta H}^{2}\gamma^{6} }}^{\frac{1}{2}}
\nonumber\\
&\qquad+ c \, \cgam^{2}\intM{\norm{\nabla A^{o}}^{2}\gamma^{2}}
       + c\oo{\intM{\norm{\nabla A^{o}}^{2}\gamma^{2} }}^{2}\,.\label{Tracefreelemma2,5}
\end{align}
Combining $\oo{\ref{Tracefreelemma2,4}},\oo{\ref{Tracefreelemma2,5}}$ and using the Cauchy-Schwarz inequality then yields
\begin{align*}
\intM{\norm{\nabla A^{o}}^{2}\gamma^{2}}+\intM{\norm{A^{o}}^{2}H^{2}\gamma^{2}}
&\leq\oo{c\norm{A^{o}}_{2,\cc{\gamma>0}}+\frac{1}{2}}\intM{\norm{\nabla A^{o}}^{2}\gamma^{2}}
\\&\qquad
+ c\Ao{\frac{4}{3}}\oo{\intM{\norm{\nabla\Delta H}^{2}\gamma^{6} }}^{\frac{1}{3}}
+ c\, \cgam^{2}\Ao{2}.
\end{align*}
For $\varepsilon_{0}$ sufficiently small we may then absorb and multiply out to
give $\oo{\ref{Tracefreelemma2,1}}$.
Using $\oo{\ref{Tracefreelemma2,1}},\oo{\ref{Tracefreelemma2,5}}$ and the
Cauchy-Schwarz inequality yields $\oo{\ref{Tracefreelemma2,2}}$.
Finally for the third statement of the lemma we combine
$\oo{\ref{Tracefreelemma2,1}}$ and Theorem \ref{MichaelSimon} with
$u=\norm{A^{o}}^{2}\gamma$ to obtain:
\begin{align*}
\intM{\norm{A^{o}}^{4}\gamma^{2}}
&\leq\cmss\oo{2\intM{\norm{\nabla A^{o}}\norm{A^{o}}\gamma}+\cgam\Ao{2}+\intM{\norm{A^{o}}^{2}\norm{H}\gamma}}^{2}\\
&\leq c\Ao{2}\oo{\intM{\norm{\nabla A^{o}}^{2}\gamma^{2}}+\intM{\norm{A^{o}}^{2}H^{2}\gamma^{2} }}+c\, \cgam^{2}\Ao{4}\\
&\leq c\Ao{\frac{10}{3}}\oo{\intM{\norm{\nabla\Delta H}^{2}\gamma^{6} }}^{\frac{1}{3}}+c\, \cgam^{2}\Ao{4}.
\end{align*}
This completes the proof.
\end{proof}

The following four lemmata are an adaptation of methods from standard energy estimates for elliptic PDE to the manifold setting.
Difficulties arise due to the geometry.
Since we are working in terms of curvature norms, such problems can be absorbed by good terms in the estimates.

\begin{lemma}\label{Tracefreelemma3}
If $f:\Sigma\rightarrow\mathbb{R}^{3}$ is an immersion satisfying $\oo{\ref{Tracefreesmallness}}$ for $\varepsilon_{0}>0$ sufficiently small, then there is a universal constant $c>0$ such that
\begin{multline}
\intM{\nablanorm{2}{A^{o}}{2}\gamma^{4}}+\intM{\norm{\nabla A^{o}}^{2}H^{2}\gamma^{4}}\\
\leq c\intM{\norm{\Delta A^{o}}^{2}\gamma^{4}}+c\Ao{\frac{2}{3}}\oo{\intM{\norm{\nabla\Delta H}^{2}\gamma^{6} }}^{\frac{2}{3}}+c \, \cgam^{4}\Ao{2} \mbox{.}\label{Tracefreelemma3,1}
\end{multline}
\end{lemma}
\begin{proof}
Using interchange of covariant derivatives:
\begin{align}
\Delta\nabla_{k}A_{lm}^{o}&=
g^{ij}\nabla_{ikj}A_{lm}^{o}+g^{ij}g^{st}\nabla_{i}\oo{R_{jkls}A_{tm}^{o}+R_{jkms}A_{lt}^{o}}\nonumber\\
&=g^{ij}\nabla_{ikj}A_{lm}^{o}+\frac{1}{4}\oo{\nabla_{l}A_{km}^{o}+\nabla_{m}A_{kl}^{o}}H^{2}\nonumber\\
&\quad -\frac{1}{4}\oo{g_{kl}\nabla_{i}\oo{A^{o}}_{m}^{i}+g_{km}\nabla_{i}\oo{A^{o}}_{l}^{i}}H^{2}+\nabla\oo{HA^{o}\star A^{o}+A^{o}\star A^{o}\star A^{o}}.\label{Tracefreelemma3,2}
\end{align}
Here $R_{ijkl}$ is the Riemann curvature tensor.
We also used \eqref{EQabove10} and
\[
H\nabla H\star A^{o}=H\nabla A^{o}\star A^{o}=\nabla\oo{HA^{o}\star A^{o}}.
\]
Performing a similar operation, we have
\begin{eqnarray}
g^{ij}\nabla_{ikj}A_{lm}^{o}
&=&\nabla_{k}\Delta A_{lm}^{o}+\frac{1}{4}\oo{\nabla_{k}A_{lm}^{o}+\nabla_{l}A_{km}^{o}+\nabla_{m}A_{kl}^{o}}H^{2}\nonumber\\
&&-\frac{1}{4}\oo{g_{kl}\nabla_{i}\oo{A^{o}}_{m}^{i}+g_{km}\nabla_{i}\oo{A^{o}}_{l}^{i}}H^{2}+\nabla\oo{HA^{o}\star A^{o}+A^{o}\star A^{o}\star A^{o}}.\label{Tracefreelemma3,3}
\end{eqnarray}
Combining $\oo{\ref{Tracefreelemma3,2}}$ and $\oo{\ref{Tracefreelemma3,3}}$ then gives us
\begin{multline}
\Delta \nabla_{k}A_{lm}^{o} = \nabla_{k}\Delta A_{lm}^{o}+\frac{1}{4}\oo{\nabla_{k}A_{lm}^{o}+2\oo{\nabla_{l}A_{km}^{o}+\nabla_{m}A_{kl}^{o} }}H^{2}\\
-\frac{1}{2}\oo{g_{kl}\nabla_{i}\oo{A^{o}}_{m}^{i}+g_{km}\nabla_{i}\oo{A^{o}}_{l}^{i}}H^{2}+\nabla\oo{HA^{o}\star A^{o}+A^{o}\star A^{o}\star A^{o}}.\label{Tracefreelemma3,4}
\end{multline}
Next tracing the term $\nabla_{l}A_{km}^{o}$ with $\nabla^{k}\oo{A^{o}}^{lm}$ gives 
\begin{align*}
\nabla^{k}\oo{A^{o}}^{lm}\nabla_{l}A_{km}^{o} &=g^{ks}g^{lt}g^{mu}\oo{\nabla_{s}A_{tu}-\frac{1}{2}g_{tu}\nabla_{s}H}\oo{\nabla_{l}A_{km}-\frac{1}{2}g_{km}\nabla_{l}H}
=\norm{\nabla A}^{2}-\frac{3}{4}\norm{\nabla H}^{2}.
\end{align*}
Substituting these results back into $\oo{\ref{Tracefreelemma3,4}}$ and taking inner products with $\nabla A^{o}$ yields
\begin{align*}
\inner{\nabla A^{o},\Delta\nabla A^{o}}
&=\inner{\nabla A^{o},\nabla\Delta A^{o}}+\frac{1}{4}\norm{\nabla A^{o}}^{2}H^{2}+\oo{\norm{\nabla A}^{2}-\frac{3}{4}\norm{\nabla H}^{2}}H^{2}\\
 &\quad -\frac{1}{4}\nabla^{k}\oo{A^{o}}^{lm}\oo{g_{kl}\nabla_{m}H+g_{km}\nabla_{l}H}H^{2}+\nabla A^{o}\star\nabla\oo{HA^{o}\star A^{o}+A^{o}\star A^{o}\star A^{o}}\\
&=\inner{\nabla A^{o},\nabla\Delta A^{o}}+\frac{5}{4}\norm{\nabla A^{o}}^{2}H^{2}-\frac{1}{2}\norm{\nabla H}^{2}H^{2}+\nabla A^{o}\star\nabla\oo{HA^{o}\star A^{o}+A^{o}\star A^{o}\star A^{o}}.
\end{align*}
Taking the negative of the previous inequality, integrating, and applying the
Cauchy-Schwarz inquality gives 
\begin{align*}
\intM{\nablanorm{2}{A^{o}}{2}\gamma^{4}}
&\leq-\intM{\inner{\nabla A^{o},\Delta\nabla A^{o}}\gamma^{4}}+4\,\cgam\intM{\norm{\nabla A^{o}}\norm{\nabla_{\oo{2}}A^{o}}\gamma^{3}}\\
&\leq-\intM{\inner{\nabla A^{o},\nabla\Delta A^{o}}\gamma^{4}}-\frac{5}{4}\intM{\norm{\nabla A^{o}}^{2}H^{2}\gamma^{4}}+\frac{1}{2}\intM{\norm{\nabla H}^{2}H^{2}\gamma^{4}}\\
 &\quad+c\intM{\norm{\nabla A^{o}}\oo{\norm{\nabla A^{o}}\norm{A^{o}}^{2}+\norm{\nabla A^{o}}\norm{A^{o}}\norm{H}}\gamma^{4}}\\
&\quad+\frac{1}{2}\intM{\nablanorm{2}{A^{o}}{2}\gamma^{4}}+8\,\cgam^{2}\intM{\norm{\nabla A^{o}}^{2}\gamma^{2}}\\
&\leq c\intM{\norm{\Delta A^{o}}^{2}\gamma^{4}}+c \, \cgam^{2}\intM{\norm{\nabla A^{o}}^{2}\gamma^{2}}-\intM{\norm{\nabla A^{o}}^{2}H^{2}\gamma^{4}}\\
&\quad+\frac{1}{2}\intM{\norm{\nabla H}^{2}H^{2}\gamma^{4}}+c\intM{\norm{\nabla A^{o}}^{2}\norm{A^{o}}^{2}\gamma^{4}}+\frac{1}{2}\intM{\nablanorm{2}{A^{o}}{2}\gamma^{4}}.
\end{align*}
We then absorb, multiply by $2$, and apply Theorem \ref{MichaelSimon} with $u=\norm{\nabla A^{o}}\norm{A^{o}}\gamma^{2}$ to obtain:
\begin{align*}
\intM{\nablanorm{2}{A^{o}}{2}\gamma^{4}}
&\leq c\intM{\norm{\Delta A^{o}}^{2}\gamma^{4}}+c \, \cgam^{2}\intM{\norm{\nabla A^{o}}^{2}\gamma^{2}}-2\intM{\norm{\nabla A^{o}}^{2}H^{2}\gamma^{4}}
+\intM{\norm{\nabla H}^{2}H^{2}\gamma^{4}}\\
&\quad+c\left( \intM{\norm{\nabla_{\oo{2}}A^{o}}\norm{A^{o}}\gamma^{2}}+\intM{\norm{\nabla A^{o}}^{2}\gamma^{2}}+\cgam\intM{\norm{\nabla A^{o}}\norm{A^{o}}\gamma} \right. \\
&\qquad \quad \left. +\intM{\norm{\nabla A^{o}}\norm{A^{o}}\norm{H}\gamma^{2}}\right)^{2}\\
&\leq c\intM{\norm{\Delta A^{o}}^{2}\gamma^{4}}+c \, \cgam^{2}\intM{\norm{\nabla A^{o}}^{2}\gamma^{2}}+\frac{1}{2}\intM{\norm{\nabla H}^{2}H^{2}\gamma^{4}}+c\oo{\intM{\norm{\nabla A^{o}}^{2}\gamma^{2} }}^{2}\\
&\quad+\oo{c\Ao{2}-1}\intM{\norm{\nabla A^{o}}^{2}H^{2}\gamma^{4}}+c\Ao{2}\intM{\nablanorm{2}{A^{o}}{2}\gamma^{4}}.
\end{align*}
Absorbing and multiplying out, using Lemma $\ref{Tracefreelemma2}$ and the Cauchy-Schwarz inequality then gives 
\begin{align*}
&\intM{\nablanorm{2}{A^{o}}{2}\gamma^{4}}+\intM{\norm{\nabla A^{o}}^{2}H^{2}\gamma^{4}}\\
&\leq c\intM{\norm{\Delta A^{o}}^{2}\gamma^{4}}+c \, \cgam^{2}\intM{\norm{\nabla A^{o}}^{2}\gamma^{2}}+c\intM{\norm{\nabla H}^{2}H^{2}\gamma^{4}}+c\oo{\intM{\norm{\nabla A^{o}}^{2}\gamma^{2} }}^{2}\\
&\leq c\intM{\norm{\Delta A^{o}}^{2}\gamma^{4}}+c\Ao{\frac{2}{3}}\oo{\intM{\norm{\nabla\Delta H}^{2}\gamma^{6} }}^{\frac{2}{3}}+c \, \cgam^{4}\Ao{2},
\end{align*}
which is the statement of the lemma.
\end{proof}

\begin{lemma}\label{Tracefreelemma4}
If $f:\Sigma\rightarrow\mathbb{R}^{3}$ is an immersion satisfying $\oo{\ref{Tracefreesmallness}}$ for $\varepsilon_{0}>0$ sufficiently small then there is a universal constant $c>0$ such that
\begin{equation}
\intM{\norm{\Delta A^{o}}^{2}\gamma^{4}}+\intM{\norm{\nabla A^{o}}^{2}H^{2}\gamma^{4}}
\leq c\Ao{\frac{2}{3}}\oo{\intM{\norm{\nabla\Delta H}^{2}\gamma^{6} }}^{\frac{2}{3}}+c\, \cgam^{4}\Ao{2}.\label{Tracefreelemma4,1}
\end{equation}
\end{lemma}
\begin{proof}
Integration by parts, estimate $\oo{\ref{TracefreeLemma1,1}}$, Theorem \ref{MichaelSimon}, and Simons' identity for $A^o$ together yield
\begin{align*}
\intM{\norm{\Delta A^{o}}^{2}\gamma^{4}}
&=\intM{\inner{\Delta A^{o},\nabla_{\oo{2}}H+\frac{1}{2}H^{2}A^{o}-\norm{A^{o}}^{2}A^{o}}\gamma^{4}}\\
&\leq\oo{\intM{\norm{\Delta A^{o}}^{2}\gamma^{4}}\cdot\intM{\nablanorm{2}{H}{2}\gamma^{4} }}^{\frac{1}{2}}-\frac{1}{4}\intM{\norm{\nabla A^{o}}^{2}H^{2}\gamma^{4}}\\
&\quad+c\intM{\norm{\nabla A^{o}}^{2}\norm{A^{o}}^{2}\gamma^{4}}+c \, \cgam^{2}\oo{\intM{\norm{A^{o}}^{2}H^{2}\gamma^{2}}+\intM{\norm{A^{o}}^{4}\gamma^{2} }}\\
&\leq\oo{\intM{\norm{\Delta A^{o}}^{2}\gamma^{4}}\cdot\intM{\nablanorm{2}{H}{2}\gamma^{4} }}^{\frac{1}{2}}-\frac{1}{4}\intM{\norm{\nabla A^{o}}^{2}H^{2}\gamma^{4}}\\
&\quad+c\Bigl(\intM{\norm{\nabla_{\oo{2}}A^{o}}\norm{A^{o}}\gamma^{2}}+\intM{\norm{\nabla A^{o}}^{2}\gamma^{2}}+\cgam\intM{\norm{\nabla A^{o}}\norm{A^{o}}\gamma}\\
&\qquad+\intM{\norm{\nabla A^{o}}\norm{A^{o}}\norm{H}\gamma^{2}}\Bigr)^{2}+c \, \cgam^{2}\oo{\intM{\norm{A^{o}}^{2}H^{2}\gamma^{2}}+\intM{\norm{A^{o}}^{4}\gamma^{2} }}\\
&\leq\frac{1}{2}\intM{\norm{\Delta A^{o}}^{2}\gamma^{4}}+\frac{1}{2}\intM{\nablanorm{2}{H}{2}\gamma^{4}}+c\oo{\intM{\norm{\nabla A^{o}}^{2}\gamma^{2} }}^{2}\\
&\quad+\oo{c\Ao{2}+\frac{1}{2}}\intM{\nablanorm{2}{A^{o}}{2}\gamma^{4}}+\oo{c\Ao{2}-\frac{1}{4}}\intM{\norm{\nabla A^{o}}^{2}H^{2}\gamma^{4}}\\
&\quad+c \, \cgam^{2}\oo{\intM{\norm{\nabla A^{o}}^{2}\gamma^{2}}+\intM{\norm{A^{o}}^{2}H^{2}\gamma^{2}}+\intM{\norm{A^{o}}^{4}\gamma^{2} }}.
\end{align*}
Once again, if $\varepsilon_{0}>0$ is sufficiently small, we can absorb and
multiply out, utilising the results of Lemmata $\ref{Tracefreelemma2}$ and
\ref{Tracefreelemma3} as well as the Cauchy-Schwarz inequality to derive:
\begin{align*}
&\intM{\norm{\Delta A^{o}}^{2}\gamma^{4}}+\intM{\norm{\nabla A^{o}}^{2}H^{2}\gamma^{4}}\\
&\leq c \, \cgam^{2}\oo{\intM{\norm{\nabla A^{o}}^{2}\gamma^{2}}+\intM{\norm{A^{o}}^{2}H^{2}\gamma^{2}}+\intM{\norm{A^{o}}^{4}\gamma^{2} }}+c\oo{\intM{\norm{\nabla A^{o}}^{2}\gamma^{2} }}^{2}
+c\intM{\nablanorm{2}{H}{2}\gamma^{4}}\\
&\leq c\Ao{\frac{2}{3}}\oo{\intM{\norm{\nabla\Delta H}^{2}\gamma^{6} }}^{\frac{2}{3}}+c \, \cgam^{4}\Ao{2}.
\end{align*}
This is the statement of the lemma.
\end{proof}

\begin{lemma}\label{Tracefreelemma5}
If $f:\Sigma\rightarrow\mathbb{R}^{3}$ is an immersion satisfying $\oo{\ref{Tracefreesmallness}}$ for $\varepsilon_{0}>0$ sufficiently small then there is a universal constant $c>0$ such that
\begin{equation*}
\intM{\nablanorm{2}{A^{o}}{2}\gamma^{4}}+\intM{\norm{\nabla A^{o}}^{2}H^{2}\gamma^{4}}
 \leq c\Ao{\frac{2}{3}}\oo{\intM{\norm{\nabla\Delta H}^{2}\gamma^{6} }}^{\frac{2}{3}}+c \, \cgam^{4}\Ao{2} \mbox{.}
\end{equation*}
\end{lemma}
\begin{proof}
Combining the results of Lemmata $\ref{Tracefreelemma3}$ and $\ref{Tracefreelemma4}$ instantly gives us the result.
\end{proof}
\begin{lemma}\label{Tracefreelemma6}
If $f:\Sigma\rightarrow\mathbb{R}^{3}$ is an immersion satisfying $\oo{\ref{Tracefreesmallness}}$ for $\varepsilon_{0}>0$ sufficiently small then there is a universal constant $c>0$ such that
\begin{equation}
\intM{\norm{\nabla A^{o}}^{2}H^{2}\gamma^{4}}+\intM{\norm{A^{o}}^{2}H^{4}\gamma^{4}}\\
 \leq c\Ao{\frac{2}{3}}\oo{\intM{\norm{\nabla\Delta H}^{2}\gamma^{6} }}^{\frac{2}{3}}+c \, \cgam^{4}\Ao{2} \mbox{.}\label{Tracefreelemma6,1}
\end{equation}
\end{lemma}
\begin{proof}
Using integration by parts, the Cauchy-Schwarz inequality, $\oo{\ref{TracefreeLemma1,1}}$ and $\oo{\ref{Prelimiaries2}}$ yields
\begin{align*}
\intM{\norm{\nabla A^{o}}^{2}H^{2}\gamma^{4}}
&\leq-\intM{\inner{A^{o},\nabla_{\oo{2}}H+\frac{1}{2}H^{2}A^{o}-\norm{A^{o}}^{2}A^{o}}H^{2}\gamma^{4}}+4\intM{\norm{\nabla A^{o}}^{2}\norm{A^{o}}\norm{H}\gamma^{4}}\\
&\quad+4\,\cgam\intM{\norm{\nabla A^{o}}\norm{A^{o}}\norm{H}\gamma^{3}}\\
&\leq\oo{32\intM{\nablanorm{2}{H}{2}\gamma^{4}}+\frac{1}{8}\intM{\norm{A^{o}}^{2}H^{4}\gamma^{4} }}-\frac{1}{2}\intM{\norm{A^{o}}^{2}H^{4}\gamma^{4}}\\
&\quad+\intM{\norm{A^{o}}^{4}H^{2}\gamma^{4}}+\oo{\frac{1}{2}\intM{\norm{\nabla A^{o}}^{2}H^{2}\gamma^{4}}+8\intM{\norm{\nabla A^{o}}^{2}\norm{A^{o}}^{2}\gamma^{4} }}\\
&\quad+\oo{\frac{1}{8}\intM{\norm{A^{o}}^{2}H^{4}\gamma^{4}}+32\, \cgam^{2}\intM{\norm{\nabla A^{o}}^{2}\gamma^{2} }}\,.
\end{align*}
Absorbing and applying Theorem \ref{MichaelSimon} twice with $u=\norm{\nabla
A^{o}}\norm{A^{o}}\gamma^{2},u=\norm{A^{o}}^{2}\norm{H}\gamma^{2}$ and the
Cauchy-Schwarz inequality, and using Lemmata $\ref{Tracefreelemma2}$ and
$\ref{Tracefreelemma5}$ gives
\begin{align*}
&\frac{1}{2}\intM{\norm{\nabla A^{o}}^{2}H^{2}\gamma^{4}}+\frac{1}{4}\intM{\norm{A^{o}}^{2}H^{4}\gamma^{4}}\\
&\leq32\oo{\intM{\nablanorm{2}{H}{2}\gamma^{4}}+\cgam^{2}\intM{\norm{\nabla A^{o}}^{2}\gamma^{2} }}+8\intM{\norm{\nabla A^{o}}^{2}\norm{A^{o}}^{2}\gamma^{4}}
 +\intM{\norm{A^{o}}^{4}H^{2}\gamma^{4}}\\
&\leq32\oo{\intM{\nablanorm{2}{H}{2}\gamma^{4}}+\cgam^{2}\intM{\norm{\nabla A^{o}}^{2}\gamma^{2} }}\\
&\quad+8\,\cmss\Bigl(\intM{\norm{\nabla_{\oo{2}}A^{o}}\norm{A^{o}}\gamma^{2}}+\intM{\norm{\nabla A^{o}}^{2}\gamma^{2}}+2\,\cgam\intM{\norm{\nabla A^{o}}\norm{A^{o}}\gamma}\\
&\quad+\intM{\norm{\nabla A^{o}}\norm{A^{o}}\norm{H}\gamma^{2}}\Bigr)^{2}+\cmss\Bigl(2\intM{\norm{\nabla A^{o}}\norm{A^{o}}\norm{H}\gamma^{2}}\\
&\quad+\intM{\norm{A^{o}}^{2}\norm{\nabla H}\gamma^{2}}+2\,\cgam\intM{\norm{A^{o}}^{2}\norm{H}\gamma}+\intM{\norm{A^{o}}^{2}H^{2}\gamma^{2}}\Bigr)^{2}\\
&\leq32\oo{\intM{\nablanorm{2}{H}{2}\gamma^{4}}+\cgam^{2}\intM{\norm{\nabla A^{o}}^{2}\gamma^{2} }}+c\Ao{2}\Bigl(\intM{\norm{\nabla_{\oo{2}}A^{o}}^{2}\gamma^{4}}\\
&\quad+\cgam^{2}\intM{\norm{\nabla A^{o}}^{2}\gamma^{2}}+\cgam^{2}\intM{\norm{A^{o}}^{2}H^{2}\gamma^{4}}\Bigr)+c\oo{\intM{\norm{\nabla A^{o}}^{2}\gamma^{2} }}^{2}\\
&\quad+c\Ao{2}\oo{\intM{\norm{\nabla A^{o}}^{2}H^{2}\gamma^{4}}+\intM{\norm{A^{o}}^{2}H^{4}\gamma^{4} }}\\
&\leq c\Ao{\frac{2}{3}}\oo{\intM{\norm{\nabla\Delta H}^{2}\gamma^{6} }}^{\frac{2}{3}}+c\cgam^{4}\Ao{2}\\
&\quad+c\Ao{2}\oo{\intM{\norm{\nabla A^{o}}^{2}H^{2}\gamma^{4}}+\intM{\norm{A^{o}}^{2}H^{4}\gamma^{4} }}.
\end{align*}
Absorbing and multiplying out then gives the statement of the lemma.
\end{proof}

Now we are in a position to prove the pointwise estimate.

\begin{theorem}\label{Tracefreetheorem1OLD}
Let $f:\Sigma\rightarrow\mathbb{R}^{3}$ be an immersion satisfying $\oo{\ref{Tracefreesmallness}}$ for $\varepsilon_{0}>0$ sufficiently small. Then there is a universal constant $c>0$ such that
\begin{equation}
\llll{A^{o}}_{\infty,\cc{\gamma=1}}^{6}\leq c\Ao{4}\oo{\intM{\norm{\nabla\Delta H}^{2}\gamma^{6}}+\cgam^{6}}\mbox{.}\label{Tracefreetheorem1,1}
\end{equation}
\end{theorem}
\begin{proof}
Using the Sobolev inequality from Theorem $\ref{EvolutionTheorem1}$ we have
\[
\llll{\Psi}_{\infty}^{6}\leq C\llll{\Psi}_{2}^{2}\oo{\llll{\nabla\Psi}_{4}^{4}+\llll{\Psi H}_{4}^{4}}
\]
for any function $\Psi$ defined on $\Sigma$. Letting $\Psi=\norm{A^{o}}\gamma^{\frac{3}{2}}$ it follows that
\begin{equation}
\llll{A^{o}}_{\infty,\cc{\gamma=1}}^{6}\\
\leq c\Ao{2}\oo{\intM{\norm{\nabla A^{o}}^{4}\gamma^{6}}+\intM{\norm{A^{o}}^{4}H^{4}\gamma^{6}}+\cgam^{4}\intM{\norm{A^{o}}^{4}\gamma^{2} }}.\label{Tracefreetheorem1,2}
\end{equation}
Since the last term was taken care of in Lemma $\ref{Tracefreelemma2}$, we just need to look at the first two terms. For the first we apply Theorem \ref{MichaelSimon} with $u=\norm{\nabla A^{o}}^{2}\gamma^{3}$ and the results of Lemma $\ref{Tracefreelemma2}$ and Lemma $\ref{Tracefreelemma6}$:
\begin{align*}
\intM{\norm{\nabla A^{o}}^{4}\gamma^{6}}
&\leq c\oo{\intM{\norm{\nabla_{\oo{2}}A^{o}}\norm{\nabla A^{o}}\gamma^{3}}+\cgam\intM{\norm{\nabla A^{o}}^{2}\gamma^{2}}+\intM{\norm{\nabla A^{o}}^{2}\norm{H}\gamma^{3} }}^{2}\\
&\leq c\intM{\norm{\nabla A^{o}}^{2}\gamma^{2}}\cdot\Bigl(\intM{\nablanorm{2}{A^{o}}{2}\gamma^{4}}+\intM{\norm{\nabla A^{o}}^{2}H^{2}\gamma^{4}}+\cgam^{2}\intM{\norm{\nabla A^{o}}^{2}\gamma^{2}}\Bigr)\\
&\leq c\Ao{2}\intM{\norm{\nabla\Delta H}^{2}\gamma^{6}}+c \, \cgam^{6}\Ao{2}.
\end{align*}

We approach the second term in $\oo{\ref{Tracefreetheorem1,2}}$ similarly, this time utilising Theorem $\ref{MichaelSimon}$ with $u=\norm{A^{o}}^{2}H^{2}\gamma^{3}$:
\begin{align*}
\intM{\norm{A^{o}}^{4}H^{4}\gamma^{6}}
&\leq c\Bigl(\intM{\norm{\nabla A^{o}}\norm{A^{o}}H^{2}\gamma^{3}}+\intM{\norm{A^{o}}^{2}\norm{\nabla H}\norm{H}\gamma^{3}}+\cgam\intM{\norm{A^{o}}^{2}H^{2}\gamma^{2}}\\
&\qquad+\intM{\norm{A^{o}}^{2}\norm{H}^{3}\gamma^{3}}\Bigr)^{2}\\
&\leq c\intM{\norm{A^{o}}^{2}H^{4}\gamma^{4}}\cdot\oo{\intM{\norm{\nabla A^{o}}^{2}\gamma^{2}}+\intM{\norm{A^{o}}^{2}H^{2}\gamma^{2}}+\cgam^{2}\Ao{2}}\\
&\quad +c\intM{\norm{\nabla H}^{2}H^{2}\gamma^{4}}\cdot\intM{\norm{A^{o}}^{4}\gamma^{2}}\\
&\leq c\Ao{2}\intM{\norm{\nabla\Delta H}^{2}\gamma^{6}}+c \, \cgam^{6}\Ao{2}.
\end{align*}
Substituting these into $\oo{\ref{Tracefreetheorem1,2}}$ and again utilising Lemma $\ref{Tracefreelemma2}$ and the Cauchy-Schwarz inequality gives the desired result.
\end{proof}

Choosing a particular cutoff function gives Theorem \ref{Tracefreetheorem1} from Theorem \ref{Tracefreetheorem1OLD}.

\end{section}

\begin{section}{Preserved closeness to the sphere}
In this section we show that any immersion with small initial
energy (in the sense of $\oo{\ref{Theorem1,1}}$) evolving under \eqref{trilaplacianIntro1}
drives the energy monotonically toward zero.
As the energy is in a sense the average deviation of the initial
immersion from an embedded round sphere, the
result will essentially show that an immersion with initially small energy will
become more spherical under the flow.

We will first state a result of Li and Yau \cite{LiYau1}.
This result will later allow us, under a similar small energy condition, to
establish that $f\oo{\Sigma,t},t\in\left[0,T\right)$ is a one-parameter family
of embeddings. 

\begin{theorem}[$\text{\cite[Theorem 6]{LiYau1}}$]\label{LiYauTheorem1}
If an immersion $f:\Sigma\rightarrow\mathbb{R}^{3}$ has the property that
\begin{equation}
\frac{1}{4}\intM{H^{2}}<8\pi,\label{LiYauTheorem1,1}
\end{equation}
then $f$ is an embedding.
\end{theorem}

\begin{theorem}[Preserved Closeness to Spheres]\label{Preservedsmallnesstheorem1}
Let $f:\Sigma\times[0,T)\rightarrow\mathbb{R}^{3}$ satisfy \eqref{trilaplacianIntro1}.
Then there exists an $\varepsilon_{0}>0$ such that if
\begin{equation}
\intM{\norm{A^{o}}^{2}}\Big|_{t=0}\leq\varepsilon_{0}<8\pi\label{Initialsmallness}
\end{equation}
then for $t<T$ we have the estimate
\[
\frac{d}{dt}\intM{\norm{A^{o}}^{2}}\leq-\frac{1}{2}\intM{\norm{\nabla\Delta H}^{2}}.
\]
In particular, $\Sigma_{t}=f\oo{\Sigma,t},t\in\left[0,T\right)$ is a one-parameter family of embeddings.
\end{theorem}
\begin{proof}
By continuity, given a flow satisfying the condition
$\oo{\ref{Initialsmallness}}$, there exists a time interval
$I_{0,\delta_{1}}=\left[0,\delta_{1}\right),\delta_{1}>0$ such that for
$\tau\in I_{0,\delta_{1}},$
\begin{equation}
\intM{\norm{A^{o}}^{2}}\Big|_{t=\tau}\leq2\, \varepsilon_{0}.\label{Preservedsmallnesstheorem1,1}
\end{equation}
Let us now look at the evolution equation associated to this integral.
We first choose a normalised frame as in the proof of Lemma $\ref{TracefreeLemma1}$. In this frame, we compute 
\[
\norm{A^{o}}^{2}=\frac{1}{2}\oo{\kappa_{1}-\kappa_{2}}^{2}=\frac{1}{2}\oo{\kappa_{1}^{2}+\kappa_{2}^{2}+2\kappa_{1}\kappa_{2}-4\kappa_{1}\kappa_{2}}=\frac{1}{2}H^{2}-2K.
\]
Applying the Gauss-Bonnet formula gives
\begin{equation}
\label{EQgb}
\frac{d}{dt}\intM{\norm{A^{o}}^{2}} =\frac{1}{2}\frac{d}{dt}\intM{H^{2}}.
\end{equation}
The evolution equation for mean curvature from Lemma $\ref{EvolutionLemma1}$ now implies
\[
\frac{d}{dt}\intM{\norm{A^{o}}^{2}}=\intM{H\Delta^{3}H}+\intM{H\Delta^{2}H\norm{A^{o}}^{2}}.
\]
Performing integration by parts yields
\begin{align*}
\frac{d}{dt}\intM{\norm{A^{o}}^{2}}&\leq -\intM{\norm{\nabla\Delta H}^{2}}+\llll{A^{o}}_{\infty}^{2}\norm{\intM{H\Delta^{2}H}}\\
&=-\intM{\norm{\nabla\Delta H}^{2}}+\llll{A^{o}}_{\infty}^{2}\intM{\norm{\Delta H}^{2}}.
\end{align*}
Combining this with Lemma $\ref{Tracefreelemma2}$ and Theorem $\ref{Tracefreetheorem1}$ with $\gamma\equiv1$ then gives 
\begin{equation*}
\frac{d}{dt}\intM{\norm{A^{o}}^{2}}\leq\oo{c\llll{A^{o}}_{2}^{2}-1}\intM{\norm{\nabla\Delta H}^{2}}\leq-\frac{1}{2}\intM{\norm{\nabla\Delta H}^{2}}
\end{equation*}
for $\varepsilon_{0}$ sufficiently small. So, we know that if
$\oo{\ref{Initialsmallness}}$ holds for $\varepsilon_{0}$ sufficiently small,
then it will decrease monotonically on the interval $I_{0,\delta_{1}}$. We may
repeat this process repeatedly right up to the maximal time of existence of the
flow, allowing us deduce that the integral $\intM{\norm{A^{o}}^{2}}$ decreases
monotonically on the interval $\left[0,T\right)$. Finally, a quick check shows
that under our initial smallness condition $\oo{\ref{LiYauTheorem1,1}}$ is
satisfied, so Theorem \ref{LiYauTheorem1} applies on $[0,T)$, completing the proof.
\end{proof}

\begin{rmk}
Theorem \ref{Preservedsmallnesstheorem1} in particular implies, in view of
\eqref{EQgb}, that the Willmore energy is non-incresing under any flow
\eqref{trilaplacianIntro1} satisfying \eqref{Initialsmallness}.
\end{rmk}
\end{section}

\begin{section}{The gap lemma}

In this section we do not assume that the immersion $f:\Sigma\rightarrow\R^3$ is closed and compact; instead, we only assume that it is proper.
Our goal is to prove Theorem \ref{GapTheorem1}.
We begin with some remarks on the theorem.

\begin{rmk}
If we do not know a-priori whether or not $f$ is compact, but we do know that \eqref{GLA1} is satisfied, then we may still conclude that $f$ is either a plane or a sphere, i.e. an embedded umbilic.
\end{rmk}
\begin{rmk}
Clearly the regularity assumption that $f$ is locally $C^6$ is much more than
what is required; indeed, Corollary \ref{CYepsreg} indicates that weak
solutions in $W^{5,2}_{\text{loc}}$ satisfying \eqref{GLA1} are classical and
smooth.
\end{rmk}

\begin{proof}[Proof of the Gap Lemma]
We shall consider the compact and noncompact cases separately. In each case we
will prove that $\norm{A^{o}}^{2}\equiv0$. The compact case is relatively
straightforward:
under the assumption $\Delta^{2}H\equiv0$ integration by parts gives us
\[
\intM{\norm{\nabla\Delta H}^{2}}=-\intM{\Delta H\cdot\Delta^{2}H}=0,
\] 
forcing $\nabla\Delta H\equiv0$. Our pointwise bound from Theorem $\ref{Tracefreetheorem1}$ with $\gamma\equiv1$ then yields
\[
\llll{A^{o}}_{\infty}^{6}\leq c\llll{A^{o}}_{2}^{4}\intM{\norm{\nabla\Delta H}^{2}}=0,
\]
and so
\begin{equation}
\norm{A^{o}}\equiv0.\label{GapTheorem1,1}
\end{equation}
For the noncompact case we will have to work a little harder.
Performing integration by parts followed by using the Cauchy-Schwarz inequality
yields
\begin{equation*}
\intM{\norm{\nabla\Delta H}^{2}\gamma^{6}}\leq\frac{1}{2}\intM{\norm{\nabla\Delta H}^{2}\gamma^{6}}+18\, \cgam^{2}\intM{\norm{\Delta H}^{2}\gamma^{4}},
\end{equation*}
so that absorbing and multiplying out yields
\begin{equation}
\intM{\norm{\nabla\Delta H}^{2}\gamma^{6}}\leq36\, \cgam^{2}\intM{\norm{\Delta H}^{2}\gamma^{4}}.\label{GapTheorem1,2}
\end{equation}
Combining this with Lemma $\ref{Tracefreelemma5}$ and using the Cauchy-Schwarz inequality,
\begin{align*}
\intM{\norm{\nabla\Delta H}^{2}\gamma^{6}}&\leq c \, \cgam^{2}\left[ \Ao{\frac{2}{3}}\oo{\intM{\norm{\nabla\Delta H}^{2}\gamma^{6} }}^{\frac{2}{3}}+\cgam^{4}\Ao{2} \right]\\
&\leq\frac{1}{2}\intM{\norm{\nabla\Delta H}^{2}\gamma^{6}}+c \, \cgam^{6}\Ao{2},
\end{align*}
implying that
\begin{equation}
\intM{\norm{\nabla\Delta H}^{2}\gamma^{6}}\leq c \, \cgam^{6}\Ao{2}.\label{GapTheorem1,4}
\end{equation}
Now choose $\gamma$ such that
\[
\gamma\oo{p}=\Phi\oo{\frac{1}{\rho}\norm{f\oo{p} }},
\]
where $\Phi\in C^{1}\oo{\mathbb{R}}$ and
\[
\chi_{B_{\frac{1}{2}}\oo{0}}\leq\Phi\leq\chi_{B_{1}\oo{0}}.
\]
With this choice for $\gamma$ we have $\cgam=\frac{c}{\rho}$ for some universal
constant $c>0$.
Taking $\rho\rightarrow\infty$ in $\oo{\ref{GapTheorem1,4}}$ yields
\[
\intM{\norm{\nabla\Delta H}^{2}}\leq c\liminf_{\rho\to\infty}\frac{1}{\rho^{6}}\int_{f^{-1}\oo{B_{\rho}\oo{0} }}{\norm{A^{o}}^{2}\,d\mu}=0.
\]
Here we have used the smallness assumption \eqref{GLA1} in the last step.
This tells us that $\nabla\Delta H\equiv0$, from which we can utilise our
pointwise bound from Theorem $\ref{Tracefreetheorem1}$ with $\rho\to\infty$ to
obtain
\begin{equation}
\llll{A^{o}}_{\infty}^{6}\leq c\liminf_{\rho\to\infty}\frac{1}{\rho^{6}}\int_{f^{-1}\oo{B_{\rho}\oo{0} }}{\norm{A^{o}}^{2}\,d\mu}=0,\nonumber 
\end{equation}
which again implies
\begin{equation}
\norm{A^{o}}\equiv0.\label{GapTheorem1,6}
\end{equation}

By observing $\oo{\ref{GapTheorem1,1}}$ and $\oo{\ref{GapTheorem1,6}}$ we
conclude that, compact or not, we must have $A^{o}\equiv0$ on $\Sigma$.
In terms of principal curvatures at any particular point, this means that
\[
\norm{A^{o}}^{2}=\frac{1}{2}\oo{\kappa_{1}-\kappa_{2}}^{2}\equiv0\text{  on  }\Sigma.
\]
Let us note that since $f$ is of class $C^6_{\text{loc}}$, non-smooth unions of
planes and spheres are not possible.
A classical result from Codazzi then tells us that $f(\Sigma)$ must be a standard flat
plane or standard round sphere.
We conclude that if $f\oo{\Sigma}$ is a round sphere if
compact, and a flat plane if noncompact.
\end{proof}
\end{section}

\begin{section}{Construction of the blowup}

Let $f:\Sigma\times\left[0,T\right)\rightarrow\mathbb{R}^{3}$ satisfy \eqref{trilaplacianIntro1}.
We make the definition
\begin{equation}
\label{EQcc}
\Kappa\oo{p,\tau}=\sup_{x\in\mathbb{R}^{3}}\int_{f^{-1}\oo{B_{r}\oo{x} }}{\norm{A}^{2}\,d\mu}\Big|_{t=\tau},
\end{equation}
and pick an arbitary decreasing sequence of radii $\left\{r_{j}\right\}\searrow0$.
We assume that the curvature concentrates in then sense that for each $j$
\[
t_{j}:=\inf\left\{t\geq0:\Kappa\oo{r_{j},t}>\varepsilon_{1}\right\}<T.
\]
Here $\varepsilon_{1}:=\frac{\varepsilon_{0}}{C},$ where $\varepsilon_{0},C$
are the same constants as in the Lifespan Theorem. Note that by construction,
$\left\{t_{j}\right\}$ is a monotonically increasing sequence. By the
definition of $t_{j}$, this means that
\begin{equation}
\int_{f^{-1}\oo{B_{r_{j}}\oo{x} }}{\norm{A}^{2}\,d\mu}\Big|_{t=t_{j}}\leq\varepsilon_{1}\text{  for any  }x\in\mathbb{R}^{3}.\label{Blowup1}
\end{equation}
However, for each $j$ it is possible to find a point $x_{j}\in\mathbb{R}^{3}$ such that 
\begin{equation}
\int_{f^{-1}\overline{\oo{B_{r_{j}}\oo{x_{j} } } }}{\norm{A}^{2}\,d\mu}\Big|_{t=t_{j}}\geq\varepsilon_{1}.\label{Blowup2}
\end{equation}
To see this, we can take a sequence $\left\{v_j\right\}\nearrow\infty,v_j>1$ and
consider the times $\tau_{j}=t_{j}+v_j^{-1}$ and radii
$\lambda_{j}=r_{j}+v^{-2}$.
By construction we have
\[
\tau_{j}\searrow t_{j}\text{  and  }\lambda_{j}\nearrow r_{j}\text{  as  }j\to\infty.
\]
Then by continuity, the definition of $t_{j}$, and the fact that $\tau_{j}>t_{j}$, for each $v$ we can choose a point $x_{j+v^{-1}}\in\mathbb{R}^{3}$ such that
\[
\int_{f^{-1}\oo{B_{\lambda_{j}}\oo{x_{j+v^{-1} } } }}{\norm{A}^{2}\,d\mu}\Big|_{t=\tau_{j}}\geq\varepsilon_{1},
\]
so that taking $v\to\infty$ yields $\oo{\ref{Blowup2}}$.

Now consider the sequence of rescaled immersions 
\begin{equation*}
f_{j}: \Sigma\times\left[-r_{j}^{-6}t_{j},r_{j}^{-6}\oo{T-{t_{j} }}\right)\rightarrow\mathbb{R}^{3}
\text{, where }f_j(p,t) = \frac{1}{r_{j}}\oo{f\oo{p,t_{j}+r_{j}^{6}t}-x_{j}}\,.
\end{equation*}
Let us define $\Kappa_{j}$ to be the value of $\Kappa$, computed as in \eqref{EQcc}, with the immersion $f_j$ in place of $f$.
That is
\[
\Kappa_j\oo{p,\tau}=\sup_{x\in\mathbb{R}^{3}}\int_{f_j^{-1}\oo{B_{r}\oo{x} }}{\norm{A_j}^{2}\,d\mu_j}\Big|_{t=\tau}\,.
\]
Quantities decorated with the subscript $j$ correspond to the immersion $f_j$.
One may check that
\[
f_{j}^{-1}\overline{\oo{B_{1}\oo{0} }}\Big|_{t=0}=f^{-1}\overline{\oo{B_{r_{j}}\oo{x_{j} } }}\Big|_{t=t_{j}},
\]
so that by $\oo{\ref{Blowup2}}$ we conclude  
\[
\int_{f_{j}^{-1}\overline{\oo{B_{1}\oo{0} } }}{\norm{A_j}^{2}\,d\mu_j}\Big|_{t=0}\geq\varepsilon_{1}.
\]
Similarly, one may check that for $\tau\leq0$:
\begin{equation}
\Kappa_{j}\oo{1,\tau}\leq\sup_{x\in\mathbb{R}^{3}}\int_{f^{-1}\oo{B_{r_{j}}\oo{x}
}}{\norm{A}^{2}\,d\mu}\Big|_{t=t_{j}+r_{j}^{6}\tau\leq t_{j}}\leq
\varepsilon_{1}.\label{Blowup3}
\end{equation}
Recall that the Lifespan Theorem tells us that the original flow $f$
will exist up until such time as the curvature concentrates, with $T$
satisfying $T\geq\frac{1}{C}$.
For the rescaled immersions, this implies
\[
r_{j}^{-6}\oo{T-t_{j}}\geq\frac{1}{C}\text{  for each  }j.
\]
Additionally, statement $\oo{\ref{lifespan2}}$ of the Lifespan Theorem tells us that for $0\leq \tau\leq\frac{1}{C}$ we have
\[
\Kappa_{j}\oo{1,\tau}\leq C\varepsilon_{1}=\varepsilon_{0}\text{  for  }0<\tau\leq\frac{1}{C}.
\]
We then apply the interior estimates from Theorem $\ref{InteriorEstimatesTheorem1}$ on the cylinder $B_{1}\oo{x}\times\left(t-1,t\right]$ to conclude that 
\begin{equation}
\llll{\nabla_{\oo{k}}A_j}_{\infty,f_j}\leq c\oo{k}\text{  for  }-r_{j}^{-6}t_{j}+1\leq t\leq\frac{1}{C}\,.\label{Blowup4}
\end{equation}
Note that the covariant derivative above is that associated with $f_j$.
Consider the sets
\[
\Sigma_j\oo{R}:=\left\{p\in\Sigma:\norm{f_{j}\oo{p}}<R\right\}=\Sigma\cap f_{j}^{-1}\oo{B_{R}}\,.
\]
Let us now show that $\mu_{j}\oo{\Sigma\oo{R}}\leq c\oo{R}$ for any $R>0$.
Recall the following monotonicity formula due to Simon.
\begin{lemma}[$\text{\cite[Equation (1.3)]{Simon1}}$]
\label{SimonLemma}
Let $f:\Sigma\rightarrow\R^3$ be an immersed surface.
Then for $0 < \sigma \le \rho < \infty$ we have
\[
\frac{|\Sigma_\sigma|}{\sigma^2}
 \le c\bigg( \frac{|\Sigma_\rho|}{\rho^2} + \int_{\Sigma_\rho} H^2\,d\mu \bigg)\,.
\]
\end{lemma}
By Theorem $\ref{Preservedsmallnesstheorem1}$, the Willmore energy is
monotonically decreasing along the flow, and so Lemma \ref{SimonLemma}
with $\sigma = R$, $\rho = R\sqrt{2c}$ gives
\[
\frac{\mu_j(\Sigma_j(R))}{R^2}
 \le
   \frac{c\mu_j(\Sigma_j(R))}{2cR^2} 
 + c\int_{\Sigma_\rho} H^2\,d\mu\,.
\]
Absorbing yields
\begin{equation}
\frac{\mu_j(\Sigma_j(R))}{R^2}
  \leq c\oo{\varepsilon_{0}+4\pi\chi\oo{\Sigma}}.
\label{Blowup8}
\end{equation}
We now state the following compactness theorem.

\begin{theorem}[$\text{\cite[Theorem 4.2]{Kuwert1}}$]\label{BlowupTheorem1}
Let $f_{j}:\Sigma_{j}\rightarrow\mathbb{R}^{3}$ be a sequence of proper immersions, where each $\Sigma_{j}$ is a surface without boundary. For $R>0$ define
\[
\Sigma_{j}\oo{R}:=\left\{p\in\Sigma_{j}:\norm{f_{j}\oo{p}}<R\right\}=\Sigma_{j}\cap f_{j}^{-1}\oo{B_{R}}.
\]
Assume the bounds
\begin{equation}
\mu_{j}\oo{\Sigma_{j}\oo{R}}\leq c\oo{R}\text{  for any  }R>0,\text{  and }
\llll{\nabla_{\oo{k}}A}_{\infty} \leq c\oo{k}\text{  for any  }k\in\mathbb{N}_{0}
\label{BlowupTheorem1,1}
\end{equation}
hold. Then there is a proper immersion $\tilde{f}:\tilde{\Sigma}\rightarrow\mathbb{R}^{3}$ (where $\tilde{\Sigma}$ is also a $2-$manifold without boundary) such that after passing to a subsequence one has
\begin{equation}
f_{j}\circ\phi_{j}=\tilde{f}+u_{j}\text{  on  }\tilde{\Sigma}\oo{j}:=\tilde{\Sigma}\cap\tilde{f}^{-1}\oo{B_{j}},\label{BlowupTheorem1,2}
\end{equation}
satisfying the following properties:
\[
\begin{cases}
&\phi_{j}:\tilde{\Sigma}\oo{j}\rightarrow U_{j}\subset\Sigma_{j}\text{  is a diffeomorphism},\\
&\Sigma_{j}\oo{R}\subset U_{j}\text{  if  }j\geq j\oo{R},\\
&u_{j}\in C^{\infty}\oo{\tilde{\Sigma}\oo{j},\mathbb{R}^{3}}\text{  is normal along  }\tilde{f}\oo{\tilde{\Sigma}},\text{  and}\\
&\llll{\tilde{\nabla}_{\oo{k}}u_{j}}_{\infty,\Sigma_{j}}\to0\text{  as  }j\to\infty\text{  for any  }k\in\mathbb{N}_{0}.
\end{cases}
\]
\end{theorem}
The theorem says that on any ball $B_{R},R>0$, for sufficiently large $j$ our
sequence of immersions can be written as a normal graph $\tilde{f}+u_{j}$ over
our limit immersion $\tilde{f}$ with small norm (after reparametrising with the
diffeomorphisms $\phi_{j}$).

Our work above, in particular \eqref{Blowup4} and \eqref{Blowup8} imply that we
may apply the compactness theorem above to the sequence $\{f_j\}$ in order to
extract a convergent subsequence asymptotic to a proper immersion
$\tilde{f}_{0}:\tilde{\Sigma}\rightarrow\mathbb{R}^{3}$.
We let $\phi_{j}:\tilde{\Sigma}\oo{j}\rightarrow U_{j}\subset\Sigma$ be a
sequence of diffeomorphisms as in $\oo{\ref{BlowupTheorem1,2}}$. Then the
reparametrisation
\begin{equation}
f_{j}\oo{\phi_{j},\cdot}:\tilde{\Sigma}\oo{j}\times\cc{0,C^{-1}}\rightarrow\mathbb{R}^{3}\label{Blowup6}
\end{equation}
satisfies \eqref{trilaplacianIntro1}. Also, by $\oo{\ref{BlowupTheorem1,2}}$,
the reparametrised flows $\left\{f_{j}\oo{\phi_{j},\cdot}\right\}$ have initial
data
\begin{equation}
f_{j}\oo{\phi_{j},0}=\tilde{f}_{0}+u_{j}:\tilde{\Sigma}\oo{j}\rightarrow\mathbb{R}^{3}\,,\label{Blowup7}
\end{equation}
converging locally in $C^\infty$ to the immersion
$\tilde{f}_{0}:\Sigma\rightarrow\mathbb{R}^{3}$. By converting our covariant
derivatives of curvature into partial derivatives of our immersion functions
$f_{j}$, we conclude from $\oo{\ref{Blowup4}}$ that 
\[
f_{j}\oo{\phi_{j},\cdot}\rightarrow\tilde{f}
\]
locally in $C^{\infty}$ where
$\tilde{f}:\tilde{\Sigma}\times\cc{0,\frac{1}{c}}\rightarrow\mathbb{R}^{3}$ satisfies \eqref{trilaplacianIntro1}
with initial data $f_{0}$. Hence the convergence is locally smooth.
We now prove four key properties of the blowup.

\begin{theorem}\label{BlowupTheorem2}
Let $f:\Sigma\times\left[0,T\right)\rightarrow\mathbb{R}^{3}$ satisfy \eqref{trilaplacianIntro1} and
$\oo{\ref{GLA1}}$. Then the blowup $\tilde{f}$ constructed above is stationary.
\end{theorem}
\begin{proof}
By Theorem $\ref{Preservedsmallnesstheorem1}$ and the fact that each of the maps $\phi_{j}$ is a diffeomorphism, we have for $t\in\cc{0,\frac{1}{C}}$:
\begin{equation}
\int_{U_{j}}{\norm{\nabla\Delta H_{j}}^{2}\,d\mu_{j}}\Big|_{t=\tau}\leq-\frac{d}{dt}\int_{\Sigma}{\norm{A_{j}}^{2}\,d\mu_{j}}\Big|_{t=\tau}.\label{BlowupTheorem2,1}
\end{equation}
By scale invariance we have
\[
\int_{\Sigma}{\norm{A_{j}}^{2}\,d\mu_{j}}\Big|_{t=0}=\intM{\norm{A}^{2}}\Big|_{t=t_{j}}\mbox{, and }\int_{\Sigma}{\norm{A_{j}}^{2}\,d\mu_{j}}\Big|_{t=C^{-1}}=\intM{\norm{A}^{2}}\Big|_{t=t_{j}+r_j^6C^{-1}}\,.
\]
Thus integrating identity $\oo{\ref{BlowupTheorem2,1}}$ over the interval $\cc{0,\frac{1}{C}}$ then yields
\begin{equation*}
\int_{0}^{\frac{1}{C}}{\int_{\tilde{\Sigma}\oo{j}}{\norm{\nabla\Delta H_j}}^2\,d\mu_j}\,d\tau
\leq\intM{\norm{A}^{2}}\Big|_{t=t_{j}}-\intM{\norm{A}^{2}}\Big|_{t=t_{j}+r_j^6C^{-1}}\,,
\end{equation*}
so that taking $j\to\infty$ yields
\[
\lim_{j\to\infty}\int_{0}^{\frac{1}{C}}{\int_{\tilde{\Sigma}\oo{j}}{\norm{\nabla\Delta H_j} }^{2}\,d\mu_j\,d\tau}=0.
\]
Hence as $f_{j}\oo{\phi_{j},\cdot}\to\tilde{f}$ locally smoothly, we find that $\nabla\Delta H\equiv0$ with respect to the immersion $\tilde{f}$. This tells us immediately that $\tilde{f}$ satisfies
\[
\Delta^{2}H\equiv0\,.
\] 
\end{proof}

\begin{lemma}\label{BlowupLemma1}
The blowup $\tilde{f}$ constructed above is not a union of planes. 
\end{lemma}
\begin{proof}
By identity $\oo{\ref{Blowup2}}$ and construction of $\tilde{f}$, we have 
\[
\int_{\tilde{f}^{-1}\overline{\oo{B_{1}\oo{0} } }}{\norm{A}^{2}\,d\mu}\geq\varepsilon_{1}>0,
\]
This tells us that $\tilde{f}$ has a nonplanar component.
\end{proof}

\begin{lemma}\label{BlowupLemma2}
Let $\tilde{f}:\tilde{\Sigma}\rightarrow\mathbb{R}^{3}$ be the blowup
constructed above. If $\tilde{\Sigma}$ contains a compact component
$\mathcal{C}$, then $\tilde{\Sigma}=\mathcal{C}$ and $\Sigma$ is diffeomorphic
to $\mathcal{C}$.
\end{lemma}
\begin{proof}
Let $\mathcal{C}$ be the aforementioned compact component of $\tilde{\Sigma}$.
Then for sufficiently large $j$, $\phi_{j}\oo{\mathcal{C}}$ is both closed and
open (clopen) in $\Sigma$. Since $\Sigma$ is connected and
$\mathcal{C}\neq\emptyset$, we conclude that $\Sigma=\phi_{j}\oo{\mathcal{C}}$.
Hence $\Sigma$ is diffeomorphic to $\mathcal{C}$. Taking $j\nearrow\infty$ then
allows us to conclude that
\[
\mathcal{C}=\lim_{j\to\infty}\phi_{j}^{-1}\oo{\Sigma}=\lim_{j\to\infty}\tilde{\Sigma}\oo{j}=\tilde{\Sigma}\,.
\]
\end{proof}
\begin{theorem}\label{BlowupTheorem3}
Let $f:\Sigma\rightarrow\mathbb{R}^{3}$ satisfy \eqref{trilaplacianIntro1} and
$\oo{\ref{Initialsmallness}}$. Let $\tilde{f}$ be the blowup constructed above.
Then none of the components of $\tilde{f}(\tilde{\Sigma})$ are compact, and the
blowup has a component which is nonumbilic and satisfies $\Delta^{2}H\equiv0$.
\end{theorem}	
\begin{proof}
We claim that the surface area $\mu\oo{\Sigma_t}$ is uniformly bounded
away from zero. To see this, note that by Theorem $\ref{Preservedsmallnesstheorem1}$
each $f_{t}$ is an embedding for $t\in\left[0,T\right)$, and \eqref{EQvolevo}
implies that the enclosed volume of the flow is constant. Combining this with
the isoperimetric inequality we conclude that for $t$ in this time interval,
\[
\mu\oo{\Sigma_{t}}\geq\sqrt[3]{36\pi\text{Vol}\oo{\Sigma_{0} }}>0.
\]
Next assume (for the sake of contradiction) that $\tilde{f}(\tilde{\Sigma})$
has a compact component, say $\mathcal{D}$. The properness of $\tilde{f}$
implies that $\tilde{f}^{-1}\oo{\mathcal{D}}\subseteq\tilde{\Sigma}$ is also
compact. We infer from Lemma $\ref{BlowupLemma1}$ that
$\mathcal{D}=\tilde{f}(\tilde{\Sigma})$, so that $\tilde{f}(\tilde{\Sigma})$
must consist of a single compact component. Hence
\begin{equation}
|\tilde{f}(\tilde{\Sigma})|=\lim_{j\to\infty}\mu_{j}\oo{\Sigma_t}<\infty.\label{BlowupTheorem3,1}
\end{equation}
We next use the definition of the sequence of immersions $\left\{f_{j}\right\}$ to compute the measure of the blowup.
Firstly, a quick computation gives
\[
g\Big|_{t=t_{j}}=r_{j}^{2}g_{j}\Big|_{t=0}
\]
where $g_{j}$ denotes the metric induced by the immersion $f_{j}$. The measure of the blowup can then be calculated from the formula
\begin{equation*}
\mu\oo{\Sigma_t}\Big|_{t=t_{j}}=\int_{\Sigma}{\,d\mu}\Big|_{t=t_{j}}=r_{j}^{2}\int_{\Sigma}{\,d\mu_{j}}\Big|_{t=0}=r_{j}^{2}\mu_{j}\oo{\Sigma_t}\Big|_{t=0}.
\end{equation*}
Thus by our choice of $r_{j}$, we know that
\begin{equation}
\mu\oo{\Sigma_t}\Big|_{t=T}=\lim_{j\to\infty}r_{j}^{2}\mu_{j}\oo{\Sigma_t}\Big|_{t=0}=0.\label{BlowupTheorem3,2}
\end{equation}
This of course contradicts our statement above regarding the enclosed area
being strictly positive. Thus we conclude that $\tilde{f}(\tilde{\Sigma})$ has
no compact components.
Recall that the locally smooth convergence rules out non-smooth combinations of
pieces of spheres and planes.
Now Lemma $\ref{BlowupLemma1}$ tells us that there must be a component of
$\tilde{f}(\tilde{\Sigma})$ with nonzero curvature, ruling out flat planes; the
only umbilics with unbounded area.
This tells us that the component above is noncompact and nonumbillic, which is
what we wished to show. 
\end{proof}
\end{section}

\begin{section}{Global analysis of the flow}
We begin by proving that a flow \eqref{trilaplacianIntro1} with initial data satisfying \eqref{Initialsmallness} exists for all time.
\begin{proposition}\label{BlowupProposition1}
Suppose $f:\Sigma\times\left[0,T\right)\rightarrow\mathbb{R}^{3}$ satisfies \eqref{trilaplacianIntro1} and \eqref{Initialsmallness}.
Then $T=\infty$.
\end{proposition}
\begin{proof}
Suppose that $T<\infty$.
Then Theorem $\ref{LifespanTheorem}$ implies that curvature must
at time $T$.
Next, by Theorem $\ref{Preservedsmallnesstheorem1}$ we know that $A^{o}$
remains small in $L^{2}$ on the time interval $\left[0,T\right)$.
We construct a blowup $\tilde{f}$ as in Section 9.
We proved that the blowup $\tilde{f}$ satisfies the hypothesis of Theorem
$\ref{GapTheorem1}$, and so we conclude that $\tilde{f}(\tilde{\Sigma})$ is
either a plane or a sphere.
But Theorem $\ref{BlowupTheorem3}$ tells us that $\tilde{f}(\tilde{\Sigma})$ is
noncompact with a nonumbilic component, and so we have reached a contradiction.
We conclude that $T=\infty$.
\end{proof}

Let us now identify a subsequence of times along which the flow is asymptotic to a round sphere.

\begin{lemma}\label{ConvergenceLemma1}
Let $f:\Sigma\times\left[0,T\right)\rightarrow\mathbb{R}^{3}$ satisfy \eqref{trilaplacianIntro1}
and $\oo{\ref{Initialsmallness}}$.
Then for any sequence $t_{j}\nearrow\infty$ we may choose
$x_{j}\in\mathbb{R}^{3},\phi_{j}\in\text{Diff}\oo{\Sigma,\mathbb{R}^{3}}$ such
that after passing to a subsequence, the immersions $f\oo{\phi_{j},t}-x_{j}$
converge in $C^\infty$ to an embedded round sphere.
\end{lemma}
\begin{proof}
Proposition $\ref{BlowupProposition1}$ implies that $T=\infty$.
Let $p\in\Sigma$ and set $x_{j}=f\oo{p,t_{j}}$.
By Theorem $\ref{InteriorEstimatesTheorem1}$, we conclude that for every $t_{j}$ we have
\[
\llll{\nabla_{\oo{k}}A}_{\infty}\Big|_{t=t_{j}}\leq c\oo{k}.
\]
By \eqref{EQareaevo}, we know that if $\mu\oo{\Sigma_{0}}$ it remains so, that is:
\begin{equation}
\mu\oo{\Sigma}\Big|_{t=t_{j}}\leq \mu\oo{\Sigma_{0}}<\infty\,.
\label{ConvergenceLemma1,2}
\end{equation}
In particular, if we consider the sequence of immersions $f_{j}:\Sigma\to\mathbb{R}^{3}$ given by $f_{j}\oo{p,t}=f\oo{p,t_{j}}-x_{j}$ then Theorem $\ref{BlowupTheorem1}$ guarantees the existence of a proper immersion $\tilde{f}:\tilde{\Sigma}\rightarrow\mathbb{R}^{3}$ (where, of course, $\tilde{\Sigma}$ is a surface without boundary) and a sequence $\phi_{j}\in\text{Diff}\oo{\Sigma,\mathbb{R}^{3}}$ such that
\begin{equation}
f_{j}\oo{\phi_{j},t}=f\oo{\phi_{j},t_{j}}-x_{j}\rightarrow\tilde{f}\text{  as  }j\nearrow\infty.\label{ConvergenceLemma1,3}
\end{equation}
Here the convergence is in the smooth topology.
We define a new sequence of flows
$h_{j}:\tilde{\Sigma}\oo{j}\times\left[-t_{j},\infty\right)\rightarrow\mathbb{R}^{3}$
defined by
\[
h_{j}\oo{p,t}=f\oo{\phi_{j}\oo{p},t_{j}+t}-x_{j}.
\]
Then each $h_{j}$ also satisfies the interior estimates and bounded area hypothesis of Theorem $\ref{InteriorEstimatesTheorem1}$, and from $\oo{\ref{ConvergenceLemma1,3}}$ we conclude that
\[
h_{j}\oo{p,0}=f\oo{\phi_{j}\oo{p},t_{j}}-x_{j}\rightarrow\tilde{f}\text{  as  }j\nearrow\infty.
\]
That is to say, at initial time our sequence $h_{j}$ converges locally in $C^{\infty}$ to $\tilde{f}$.
Following the same line of argument as in the proof of Theorem $\ref{BlowupTheorem2}$ we conclude that 
\begin{equation*}
\int_{t_{j}}^{t_{j+1}}{\intM{\norm{\nabla\Delta H}^{2}}\,d\tau}\leq\intM{\norm{A}^{2}}\Big|_{t=t_{j}}-\intM{\norm{A}^{2}}\Big|_{t=t_{j+1}}\searrow0\text{  as  }j\nearrow\infty,
\end{equation*}
which tells us that $\tilde{f}$ satisfies $\Delta^2H \equiv 0$.
Noting that $\tilde{f}$ here is compact, Theorem $\ref{GapTheorem1}$ then
implies that $\tilde{f}$ must be a sphere.
\end{proof}

We now use a standard argument to obtain exponentially fast convergence of the family of immersions $f$ in the $C^\infty$ topology.

\begin{proposition}\label{ConvergenceProposition1}
Suppose $f:\Sigma\times\left[0,T\right)\rightarrow\mathbb{R}^{3}$ satisfies
\eqref{trilaplacianIntro1} and $\oo{\ref{Initialsmallness}}$.
Then for all $t$ sufficiently large
\[
\llll{\nabla_{\oo{k}}A}_{\infty}\leq c_{k}e^{-\xi t}\text{  and  }\llll{A^{o}}_{\infty}\leq c_{0}e^{-\xi t},
\] 
where $k\in\mathbb{N}$.
\end{proposition}
\begin{proof}
Because the flow \eqref{trilaplacianIntro1} preserves volume, the radius of the limiting sphere is
\[
\rho_{\infty}=\sqrt[3]{\frac{3 \mbox{Vol}\oo{\Sigma_{0} }}{4\pi}}.
\]
By continuity, there is a time after which $f\left( \cdot, t\right)$ remains a radial graph; we linearise the radial graph evolution corresponding to \eqref{trilaplacianIntro1} over the limiting stationary immersion and apply the principle of linearised stability, along with a result of Lunardi \cite{Lunardi1} to obtain exponential convergence to the sphere.
Specifically, for  $t$ sufficiently large, we write our surface as a radial graph over the unit $2-$sphere:
\[
f\oo{z,t}=\rho\oo{z,t}z,
\]
for $z\in\mathbb{S}^{2}$, adding a tangential diffeomorphism to
\eqref{trilaplacianIntro1} such that the parametrisation is preserved. The
components of the induced metric $g$ are
\[
g_{ij}=\rho^{2}\sigma_{ij}+\snabla_{i}\rho\snabla_{j}\rho.
\]
Here $\sigma$ is the metric on $\mathbb{S}^{2}$. We have used the identities
$z\perp\snabla_{i}z$ and $\norm{z}^{2}=1$.
We compute (see \cite{McCoy1} for details)
\[
A_{ij}=-\Phi\oo{\rho}^{-\frac{1}{2}}\oo{\rho\snabla_{ij}\rho-2\snabla_{i}\rho\snabla_{j}\rho-\sigma_{ij}\rho^{2}} \mbox{ and } g^{ij}=\rho^{-2}\oo{\sigma^{ij}-\Phi\oo{\rho}^{-1}\snabla^{i}\rho\snabla^{j}\rho},
\]
where
\[
\Phi\oo{\rho}:=\rho^{2}+\norm{\snabla\rho}^{2}
\]
and $\sDelta$ denotes the Laplace-Beltrami operator on $\mathbb{S}^{2}$. Hence
\begin{equation}
H\oo{\rho}=-\rho^{-1}\Phi\oo{\rho}^{-\frac{1}{2}}\sDelta\rho+\rho^{-1}\Phi\oo{\rho}^{-\frac{3}{2}}\snabla^{i}\rho\snabla^{j}\rho\snabla_{ij}\rho+2\Phi\oo{\rho}^{-\frac{1}{2}}+\Phi\oo{\rho}^{-\frac{3}{2}}\norm{\snabla\rho}^{2},\label{1}
\end{equation}
Now consider a variation of $\rho$ centred around the stationary spherical solution. That is, 
\begin{equation}
\rho\mapsto\rho_{\epsilon}=\rho_{\infty}+\epsilon\eta.\label{2}
\end{equation}
Taking repeated covariant derivatives of our expression for the mean curvature $H$ in $\oo{\ref{1}}$ yields
\begin{equation}
\Delta^{2}H\oo{\rho_{\epsilon}}=-\rho_{\epsilon}^{-4}\oo{\rho_{\epsilon}^{-1}\Phi\oo{\rho_{\epsilon}}^{-\frac{1}{2}}\sDelta^{3}\rho_{\epsilon}+\Phi\oo{\rho_{\epsilon}}^{-\frac{3}{2}}\sDelta^{2}\Phi\oo{\rho_{\epsilon}
}}+Q(\rho_\epsilon,\eta,\epsilon)\,,\label{3}
\end{equation}
where $Q$ satisfies $\frac{d}{d\epsilon}Q\Big|_{\epsilon=0}=0$.
Hence
\begin{align*}
\frac{d}{d\epsilon}\Delta^{2}H\oo{\rho_{\epsilon}}\Big|_{\epsilon=0}&=-\frac{d}{d\epsilon}\rho_{\epsilon}^{-4}\oo{\rho_{\epsilon}^{-1}\Phi\oo{\rho_{\epsilon}}^{-\frac{1}{2}}\sDelta^{3}\rho_{\epsilon}+\Phi\oo{\rho_{\epsilon}}^{-\frac{3}{2}}\sDelta^{2}\Phi\oo{\rho_{\epsilon} }}\Big|_{\epsilon=0}\\
&=-\rho_{\epsilon}^{-5}\Phi\oo{\rho_{\epsilon}}^{-\frac{1}{2}}\sDelta^{3}\eta\Big|_{\epsilon=0}-2\rho_{\epsilon}^{-3}\Phi\oo{\rho_{\epsilon}}^{-\frac{3}{2}}\sDelta^{2}\eta\Big|_{\epsilon=0}\\
&=-\rho_{\infty}^{-6}\oo{\sDelta^{3}\eta+2\sDelta^{2}\eta}.
\end{align*}
Hence the linearisation of \eqref{trilaplacianIntro1} about the stationary sphere solution with radius $\rho_{\infty}$ is
\begin{equation}
\frac{\partial\eta}{\partial t}=\rho_{\infty}^{-6}\oo{\sDelta^{3}\eta+2\sDelta^{2}\eta}=:\mathcal{L}\eta.\label{4}
\end{equation}
It is well-known (see \cite{Chavel} for example) that the set of eigenvalues of
the Laplacian $\sDelta$ on $\mathbb{S}^{2}$ is
\[
\sigma\oo{\sDelta}=\left\{\lambda_{l}:l\in\mathbb{N}_{0}\right\}=\left\{-l\oo{l+1}:l\in\mathbb{N}_{0}\right\}\subset\left(-\infty,0\right],
\]
with the algebraic multiplicity of each eigenvalue $\lambda_{l}$ being equal to
the dimension of the space of homogenous, harmonic polynomials of degree $l$ on
the sphere.
In particular, the multiplicities of $\lambda_{0}$ and $\lambda_{1}$ are $1$
and $3$ respectively. It is also well-known that the eigenvalue of the
$p$-times iterated Laplacian $\sDelta^{p}$ ($p\in\mathbb{N}$), is given by
\begin{equation*}
\sigma\oo{\sDelta^{p}}=\left\{\lambda_{l}^{p}:\lambda_{l}\in\sigma\oo{\sDelta}\right\}.
\end{equation*}
Hence it follows that the set of eigenvalues of the operator $\mathcal{L}$ in \eqref{4} is
\begin{equation*}
\sigma\oo{\mathcal{L}}=\left\{-\rho_{\infty}^{-6}\oo{l+2}\oo{l+1}^{2}l^{2}\oo{l-1}:l\in\mathbb{N}_{0}\right\}.\label{5}
\end{equation*}
The zero eigenvalue (corresponding to $l=0,1$) of $\mathcal{L}$ has algebraic
multiplicity $4$.
One may follow a set of steps completely analogous to \cite[pp. 1428--1430]{Escher1},
quotienting out the zero eigenvalues and proving that \cite{Simonett1} can be
applied to conclude exponentially fast convergence to $\rho_\infty$ in the
$C^{\infty}$ topology.
Converting the norms on the derivatives of our immersion function into
covariant derivatives of curvature allows us to establish the desired
exponential convergence result of the theorem.
\end{proof}

\end{section}


\begin{thebibliography}{99}

\bibitem{BWH95} B. Malcolm, M. Wilson, H. Hagen, \emph{The smoothing properties
of variational schemes for surface design}, Computer Aided Geometric Design
\textbf{12} (1995), no. 4, 381--94.

\bibitem{Chavel} I. Chavel, \emph{Eigenvalues in Riemannian Geometry} 2nd ed.,
North Holland, Amsterdam 1984 (379 pages).

\bibitem{Escher1} J. Escher, U. F. Mayer, G. Simonett, \emph{The surface
diffusion flow for immersed hypersurfaces}, Siam J. Math. Anal. \textbf{29}
(1998), no. 6, 1419--1433.

\bibitem{GomezNogueira} H. Gomez and X. Nogueira, \emph{An unconditionally
energy-stable method for the phase field crystal equation}, Comput. Methods
Appl. Mech. Engrg. \textbf{249--252} (2012), 52--61.

\bibitem{Hamilton2} M. Gage and R. S. Hamilton, \emph{The heat equation
shrinking convex plane curves}, J. Differential Geom. \textbf{23} (1986), no.
1, 69--96.

\bibitem{Hamilton1} R. S. Hamilton, \emph{Three-manifolds with positive Ricci
curvature}, J. Differential Geom. \textbf{17} (1982), 255--306.

\bibitem{Huisken1} G. Huisken and A. Polden, \emph{Geometric evolution
equations for hypersurfaces}, in Calculus of variations of geometric evolution
problems, (Cetraro, 1996), Springer-Verlag, Berlin, 1999, 45--84.

\bibitem{KUW04} S. Kubiesa, H. Ugail, M. Wilson, \emph{Interactive design using
higher order PDEs}, The Visual Computer, \textbf{20} (2004), no. 10, 682--693.

\bibitem{Kuwert1} E. Kuwert and R Sch\"{a}tzle, \emph{The Willmore flow with
small initial energy},  J. Differential Geom., \textbf{57} (2001), no. 3,
409--441.

\bibitem{Kuwert2} E. Kuwert and R. Sch\"{a}tzle, \emph{Gradient Flow for the
Willmore Functional}, Comm. Anal. Geom. \textbf{10} (2002), no. 2, 307--339.

\bibitem{LiYau1} P. Li and S.-T. Yau, \emph{A New Conformal Invariant and Its
Applications to the Willmore Conjecture and that First Eigenvalue of Compact
Surfaces}, Invent. math. \textbf{69} (1982), 269--291.

\bibitem{LiuXu} D. Liu and G. Xu, \emph{A general sixth order geometric partial
differential equation and its application in surface modeling}, Journal of
Information and Computational Science \textbf{4} (2007), 1--12.

\bibitem{Lunardi1} A. Lunardi, \emph{Analytic Semigroups and Optimal Regularity
in Parabolic Problems}, Progress in Nonlinear Differential Equations and Their
Applications, Vol. 16, Birkh\"{a}user, 1995.

\bibitem{Mantegazza1} C. Mantegazza, \emph{Smooth geometric evolutions of
hypersurfaces}, Geom. Funct. Anal. \textbf{12} (2002), 138--182.

\bibitem{Mantegazza2} C. Mantegazza and L. Martinazzi, \emph{A Note on
Quaslinear Parabolic Equations on Manifolds}, Ann. Sc. Norm. Super. Pisa Cl.
Sci. (5) \textbf{11} (2012), no. 4, 857--874. 

\bibitem{McCoy1} J. McCoy, \emph{The mixed volume preserving mean curvature
flow}, Math. Z. \textbf{246} (2004), no. 1-2, 155--166.

\bibitem{Wheeler3} J. McCoy, G. Wheeler and G. Williams, \emph{Lifespan theorem
for constrained surface diffusion flows}, Math. Z., \textbf{269} (2011),
147--178.

\bibitem{Helfrich1} J. McCoy, G. Wheeler, \emph{A classification theorem for
Helfrich surfaces}, Math. Ann. \textbf{357} (2013), 1485--1508.

\bibitem{MichaelSimon1} J.H. Michael and L. Simon, \emph{Sobolev and mean-value
inequalities on generalized submanifolds of $\mathbb{R}^{n}$}, Commun. Pure
Appl. Math. \textbf{26} (1973), 361--379.

\bibitem{Schoen} R. Schoen, \emph{Analytic aspects of harmonic maps problems},
Math. Sci. Res. Institute Pub. \textbf{2} (1984), 321--358.

\bibitem{Simon1} L. Simon, \emph{Existence of surfaces minimizing the Willmore
functional},  Comm. Anal. Geom. \textbf{1} (1993), no. 2, 281--326. 

\bibitem{Simonett1} G. Simonett, \emph{Center Manifolds for Quasilinear
Reaction-Diffusion Systems}, Diff. Integral Equations \textbf{8} (1995), no. 4,
753--796.  

\bibitem{Tosun} E. Tosun, \emph{Geometric modeling using high order
derivatives}, PhD thesis, New York University, 2008.

\bibitem{U04} H. Ugail, \emph{Spine based shape parametrisation for PDE surfaces},
Computing, \textbf{72} (2004), no. 1--2, 195--206.

\bibitem{UW05} H. Ugail and M. Wilson, \emph{Modelling of oedemus limbs and
venous ulcers using partial differential equations}, Theoretical Biology and
Medical Modelling \textbf{2} (2005), no. 1, 1--28.

\bibitem{Wheeler2} G. Wheeler, \emph{Surface diffusion flow near spheres},
Calc. Var. \textbf{44} (2012), 131--151.

\bibitem{Wheeler4} G. Wheeler, \emph{Gap phenomena for a class of fourth-order geometric differential operators on surfaces with boundary},
accepted in Proc. Amer. Math. Soc. (2013), arXiv $1302.4165$.
\end{thebibliography}
\end{document}